\documentclass[reqno, 11pt]{article}

\usepackage{graphicx}%
\usepackage{multirow}%
\usepackage{amsmath,amssymb,amsfonts}%
\usepackage{amsthm}%
\usepackage{mathrsfs}%
\usepackage[title]{appendix}%
\usepackage{xcolor}%
\usepackage{textcomp}%
\usepackage{manyfoot}%
\usepackage{booktabs}
\usepackage[utf8]{inputenc}  
\usepackage{enumitem}

\usepackage{algorithm}%
\usepackage{algorithmicx}%
\usepackage{algpseudocode}%
\usepackage{listings}%

\usepackage{fullpage}

\usepackage{enumerate}
\usepackage{comment}
\usepackage{newlfont}
\usepackage{color}
\usepackage{bbm}
\usepackage{hyperref}

\usepackage{stmaryrd}
\SetSymbolFont{stmry}{bold}{U}{stmry}{m}{n}

\usepackage{bbm}

\usepackage{stmaryrd}
\usepackage{mathrsfs}

\usepackage{fancyhdr}
\usepackage{graphicx}
\usepackage{fancybox}
\usepackage{setspace}
\usepackage{cleveref}

\newtheorem{thm}{Theorem}
\newtheorem{cor}[thm]{Corollary}
\newtheorem{lem}[thm]{Lemma}
\newtheorem{prop}[thm]{Proposition}

\theoremstyle{definition}
\newtheorem{defn}[thm]{Definition}

\theoremstyle{remark}
\newtheorem{rem}[thm]{Remark}

\newcommand{\R} {\mathbb{R}}

\newcommand{\N} {\mathbb{N}}

\newcommand{\E} {\mathbb{E}}

\newcommand{\beq}{ \begin{equation} }
\newcommand{\eeq}{ \end{equation} }

\renewcommand{\P}{\mathbb{P}}

\numberwithin{equation}{section} 
\numberwithin{thm}{section}

\begin{document}
\include{amsthm_sc}

\title{Heavy-Tailed Mixed p-Spin Spherical Model: Breakdown of Ultrametricity and Failure of the Parisi Formula}

 \author{Taegyun Kim\footnote{Department of Mathematical Sciences, KAIST, Daejeon, 305701, Korea \newline email: \texttt{ktg11k@kaist.ac.kr}}}



\maketitle

\begin{abstract}
We prove that the two cornerstones of mean-field spin glass theory---the Parisi variational formula and the ultrametric organization of pure states---break down under heavy-tailed disorder. For the mixed spherical \(p\)-spin model whose couplings have tail exponent \(\alpha<2\), we attach to each \(p\) an explicit threshold \(H_p^{*}\). If any coupling exceeds its threshold, a single dominant monomial governs both the limiting free energy and the entire Gibbs measure; the resulting energy landscape is intrinsically probabilistic, with a sharp failure of ultrametricity for \(p\ge4\) and persistence of only a degenerate 1-RSB structure for \(p\le3\). When all couplings remain below their thresholds, the free energy is \(O(n^{-1})\) and the overlap is near zero, resulting in a trivial Gibbs geometry. For \(\alpha<1\) we further obtain exact fluctuations of order \(n^{1-p}\). Our proof introduces \emph{Non-Intersecting Monomial Reduction} (NIMR), an algebraic--combinatorial technique that blends convexity analysis, extremal combinatorics and concentration on the sphere, providing the first rigorous description of both regimes for heavy-tailed spin glasses with \(p\ge3\).
\end{abstract}

\tableofcontents

\section{Introduction}
Spin glass theory emerged in the 1970s to explain the puzzling magnetic
behavior of dilute alloys such as CuMn and AuFe
\cite{Cannella1972}.  A family of mean field models soon became the
standard testing ground: the Edwards--Anderson model
\cite{Edwards1975Theory}, the Sherrington--Kirkpatrick (SK) model
\cite{Sherrington1975Solvable}, its spherical analogue
\cite{Kosterlitz1976Spherical}, Derrida’s Random Energy Model (REM)
\cite{Derrida1980REM}, and heavy-tailed L\'evy variants
\cite{cizeau1994theory,janzen2008replica,janzen2010thermodynamics}.

\medskip
\noindent\textbf{The Parisi picture and its limits.}
Parisi’s replica-symmetry-breaking (RSB) variational formula
\cite{parisi1979infinite} and the ultrametricity \cite{mezard1984nature}, later proved with mathematical rigor by Talagrand
\cite{Talagrand2006Parisi} and Panchenko
\cite{Panchenko2013Ultrametricity} forged the modern picture of
mean-field spin glass theory. Due to their success, most of the rigorous and heuristic work has 
relied on this tree-like ultrametric organization of the Gibbs measure. 

\medskip
\noindent\textbf{Heavy tails: a fundamental difference.}
This paper shows that this common belief for mean-field spin glass theory: Parisi variational formula and ultrametricity break down when the
coupling distribution has a tail exponent $0<\alpha<2$. In the mixed
$p$-spin spherical model, we prove:
\begin{itemize}
    \item Ultrametricity remains below the threshold but \emph{fails} in the dominant regime, giving the first counter-example of ultrametricity and the sharp threshold to break the ultrametricity. 
    \item If there exists an interaction above the threshold $H_p^*$, then only one \emph{dominant} monomial governs the free energy and the energy landscape. If such a monomial comes from $p\ge4$, it breaks ultrametricity. Otherwise, it maintains the ultrametricity as a degenerate 1-RSB form.
  \item A new normalization of order
  $n^{-(p-2)/2}b_{n,p}^{-1}$—with
  \(
    b_{n,p}=\inf\{t:\P(|H|>t)<\binom{n}{p}^{-1}\}
  \)
  (see Theorem~\ref{GSE})—is necessary and sufficient to obtain an
  $O(n)$ ground state energy.
  \item Contrast to usual Gaussian spin glass theory, it does not exhibit Parisi type variational formula and the free energy and its energy landscape converge in distribution to random variables.
\end{itemize}

\medskip
\noindent\textbf{Context and related works.}
Heavy-tailed random matrices already exhibit Poissonian edge statistics
and eigenvector localization-delocalization phase transition
\cite{auffinger2009poisson,soshnikov2004poisson,arous2008spectrum,
aggarwal2021goe,aggarwal2021eigenvector,aggarwal2022mobility}.  In the
spin-glass setting, rigorous results were previously confined to
two-body interactions using random matrix techniques
\cite{kim2024fluctuations} or heavy-tailed extensions of the SK
formula \cite{chen2025some,jagannath2024existence}.  Higher-order
interactions remained out of reach because existing tools rely
crucially on Gaussian tails (e.g., Stein's lemma, Wick's calculus).  Our
method—\emph{Non-Intersecting Monomial Reduction (NIMR)}—bypasses those
obstacles and applies to \emph{every} mixed $p$-spin Hamiltonian.

Parallel efforts to map the energy landscape of \emph{Gaussian}
spherical models through complexity counts, band decomposition, and
TAP equations
\cite{auffinger2013random,subag2017complexity,subag2017geometry,
baik2016fluctuations,arous2024shattering} using random matrix theory do not transfer directly to
heavy-tailed disorder. The present work supplies both the correct
scaling and new probabilistic tools for that extension.
\subsection{Main result}
Before stating our main result, we need to define our heavy tail model. We use different normalization and this will be justified in Theorem \ref{GSE} by showing scale of Ground State Energy.
\begin{defn}[Hamiltonian]
    We consider the indices for the case of $i_1< i_2<\cdots< i_p$. Each $H'_{i_1,...,i_p}$ is an i.i.d. copy of the heavy-tailed random variable $H$ with tail exponent $\alpha<2$. We define 
    the Hamiltonian of our mixed $p$ spin spherical model with heavy tail interaction  as 
    \[
    H_{n}(\sigma)= \sum_{p\geq 2} \alpha(p)H_{n,p}(\sigma)
\]
where 
    \[b_{n,p}=\inf\{t:\P(|H|>t)<\binom{n}{p}^{-1}\} \text{ and }H_{n,p}=\sum_{i_1<\cdots<i_p } H'_{i_1,...,i_p}n^{-(p-2)/2}b_{n,p}^{-1}\sigma_{i_1}...\sigma_{i_p}.\]

We define random variables after the suitable normalization as $H_{i_1,...,i_p}=H_{i_1,...,i_p}'b_{n,p}^{-1}$. 
We also use the notation $H_{I,p}=H_{i_1,...,i_p}$ for $I=\{i_1,...,i_p\}$ and for each $p$, we define its absolute value ordered statistics as $|H_{1,p}|\geq|H_{2,p}|\geq...\geq|H_{\binom{n}{p},p}|$. See Section 2 for more detail to basic notions.  
\end{defn}
We also define basic notions about its spin configurations. 
\begin{defn} The spin variable $(\sigma_1,...,\sigma_n)$ is uniformly distributed on the sphere \[S_n=\{(\sigma_1,...,\sigma_n)|\sum_{i=1}^n \sigma_i^2=n\}.\]
Now we write $\E$ as the expectation over the uniform measure on this sphere.
 We define the partition function and the Gibbs measure $Z_n$ and $G_n$, respectively: \[
    Z_n=\E[\exp(\beta H_n(\sigma)) ], \quad G_n(\sigma)= \exp (\beta H_n(\sigma) )/Z_n
    \]
    where $\beta$ is the inverse temperature.
    We also define free energy $F_n$ and ground state energy (\emph{GSE}) as
    \[
    F_n=\frac{1}{n}\log Z_n,\; \text{GSE}= \max_{\sigma}H_n(\sigma).
    \]
    The overlap is the inner product \(
        R_{1,2}=\frac{1}{n}\sum_{i=1}^n\sigma^1_i\sigma^2_i
    \) where $\sigma^1= (\sigma^1_1,...,\sigma^1_n)$, $\sigma^2=(\sigma^2_1,...,\sigma^2_n)$ are independently sampled from the Gibbs measure.
    We also define the expectation over the Gibbs measure as$\langle \ \rangle $. This definition also covers the expectation of the measure $G_n^{\otimes p}$.
\end{defn}
We also define the dominant interaction to express the above threshold regime. The intuition of a dominant interaction is to find such a pair:
\[
\E[\exp(\beta\alpha(p)H_{I,p}\sigma_1...\sigma_pn^{-(p-2)/2})]=\max_{J,q} \E[\exp(\beta\alpha(p)H_{J,q}\sigma_1...\sigma_qn^{-(q-2)/2})].
\]
Before we elaborate on our main theorems, we need to define their conditional events. We express our main theorems using these conditional events.
\begin{defn}
    We define \[
    \mathcal{F}_1=\{ |\beta \alpha(p) H_{1,p}|<H_p^*\text{ for all } p\}
    \]
    and
    \[\mathcal{F}_p= \{|\beta \alpha(p) H_{1,p}|>H_p^*,f_q(\beta \alpha(q)H_{1,q})<f_p(\beta \alpha(p)H_{1,p})\text{ for all } q\}.
    \]
    We also define the conditional event in $\mathcal{F}_p$ as
    \[
        \mathcal{F}(H,p)=\{H_{1,p}=H,|H|>H_p^*\text{ and }f_q(\beta \alpha(q) H_{1,q})<f_p(\beta \alpha(p) H)\text{ for all }q\}.
    \]
\end{defn}
See Definition \ref{def:threshold} for the definition of $f_p,H_p^*$. Now we can explain our main theorems.
\begin{thm}[Sampled free energy]\label{Free}
    For a mixed p-spin spherical model with heavy-tail interaction with exponent $\alpha<2$, its free energy shows the following phase transitions with high probability.
    \begin{itemize}
        \item For the event $\mathcal{F}(H,p)$ which is the dominant interaction $H_{1,p}$ has the fixed value $H$ and others are smaller than that, we have
        \[
            \frac{1}{n}\log Z_n|\mathcal{F}(H,p) = f_p(\beta\alpha(p)H)+O(n^{-\epsilon}).
        \]
        \item If $|\beta\alpha(p) H_{I,p}|<H_p^*$ holds for every index, for this event $\mathcal{F}_1$,
        \[
        \frac{1}{n}\log Z_n |\mathcal{F}_1=O(n^{-\epsilon}).
        \]
        \item Furthermore for smallest p such that $\alpha(p)\neq 0$,
         if $p=2$,
        \[
            \log Z_n|\mathcal{F}_1= -\sum_{|I|=2}\frac{1}{2}\log (1- \beta^2\alpha(2)^2{H_{I,2}^2})+O(n^{-\epsilon}).
        \]
        If $p\geq 3$,
        \[
            \log Z_n |\mathcal{F}_1= O(n^{-\epsilon}).
        \]
        Especially, if $\alpha <1$, this exhibits 
        \[
            n^{p-2}\log Z_n|\mathcal{F}_1= \frac{1}{2}\beta^2\alpha(p)^2\sum_{|I|=p} H_{I,p}^2+O(n^{-\epsilon}).
        \]
       
    \end{itemize}
\end{thm}
The author and Lee have already established the case for \( p=2 \) in \cite{kim2024fluctuations} that its fluctuation scale is $O(1)$ in the high-temperature regime (where the largest entry exceeds a certain threshold) and $O(n)$ in the low-temperature regime (where the largest entry remains below the threshold). This was achieved by utilizing the eigenvalue properties of the heavy-tail random matrices and is compatible with this result.
\begin{thm}[Sample Gibbs measure structure]\label{Gibbs measure}
    In Hamiltonian, if there is an interaction above threshold and the dominant term is of the form $\beta \alpha(p) H\sigma_1...\sigma_pn^{-(p-2)/2}$, then we have such limiting law as $n\to \infty$ this converges in distribution as
    \[
        \frac{1}{\sqrt{n}}(\sigma_1,...,\sigma_p)|\mathcal{F}(H,p)\to [
\sqrt{t}(\xi_1,...,\xi_p)|H\xi_1...\xi_p>0]
    \]
    where $\xi_i$ is a random variable $\pm 1$ with probability 1/2 and $t=\frac{2\lambda_p(2\beta \alpha(p) H)}{2p\lambda_p((2\beta \alpha(p) H))+1}$.
    Morevoer, for two replicas, its restricted overlap satisfies
    \[
        \langle(\frac{1}{n}\sum_{i=p+1}^n\sigma_{i}^1\sigma_{i}^2)^2|\mathcal{F}(H,p)\rangle=O(n^{-\epsilon})
    \]
    where $\langle \quad \rangle$ means expectation over Gibbs measure.
    Moreover, if every interaction is under the threshold, 
    then we have 
    \[
        \langle R_{1,2}^2|\mathcal{F}_1\rangle=O(n^{-\epsilon}).
    \]
For the above threshold case, the Gibbs measure is concentrated on $2^{p-1}$ connected components in $\{(\sigma_1,...,\sigma_n)|H\sigma_1...\sigma_p>0\}$. In other words,
for $x\in M_+(H,p)$, we have \[
    \langle\mathbbm{1}_{f_{\pm}(x)}|\mathcal{F}(H,p)\rangle=\frac{1}{2^{p-1}}(1-O(n^{-\epsilon}))\] and \[\sum_{x\in M_+(H,p)}\langle\mathbbm{1}_{f_{\pm}(x)}|\mathcal{F}(H,p)\rangle=1-O(n^{-\epsilon}). 
\]
Moreover, the restricted tuple $\frac{1}{\sqrt{n}}(\sigma_1,...,\sigma_p)$ is concentrated near a point in each connected component. More precisely, in each region $x\in M_+(H,p)$, we have its conditional Gibbs measure concentration:
\[
    \langle |(\sigma_1,...,\sigma_ p )-\sqrt{t}x|^2|\sigma \in f_{\pm}(x),\mathcal{F}(H,p)\rangle= O(n^{-\epsilon}).
\]
\end{thm}
See Definition \ref{def:threshold}, \ref{M+} for the notation. This gives a corollary about the RSB structure for each case.
\begin{cor}
    The above implies that if the dominant interaction comes from $p$-spin, its overlap structure has $\lfloor \frac{p}{2}\rfloor$-RSB.
\end{cor}
Furthermore, this shows the breaking of ultrametricity for a certain case which is one of the most important core structures of the spin glass theory.
\begin{thm}[Non-ultrametricity]\label{ultrametric}
    This heavy-tailed spin glass does not satisfy the ultrametricity if there is an interaction above the threshold, and the dominant term comes from $p \geq 4$. Otherwise, this still shows ultrametricity.
\end{thm}
We can also show that its scale of ground state energy is of the order $n$ and this justifies our choice of a different normalization for this heavy-tail model.
\begin{thm}[Ground state energy]\label{GSE}
    The ground state energy for mixed $p$ spin spherical model with heavy-tail interaction satisfies 
    \[
        \frac{1}{n}\text{GSE}=\max_{I,p} |\alpha(p)\beta H_I|p^{-p/2}+O(n^{-\epsilon'})
    \]
    for some $\epsilon'>0$.
\end{thm}
Due to the lack of concentration of measure, the energy landscape changes whenever we observe the interactions, in contrast to the usual Gaussian spin glass. Hence, we can express its free energy and energy landscape in a probabilistic way and that is why heavy tail spin glass does not have Parisi type formula and exhibits fundamentally different nature.
\begin{thm}[Probabilistic energy landscape]\label{prob_land}
    For $n\to \infty$ and $\alpha(p)$ profile is given, then the free energy and overlap converges to a random variable.
   For the limiting spin behavior, the number of spins for dominant interaction converges to a random variable $X$ which shows dominant term comes from $t-$spin interaction with probability $p_t$ and $X=1$ with probability $p_1$ which is the probability that every interaction is under the threshold. For this event $\mathcal{F}_t$, we have conditional limiting spin behavior. For this limiting spin behavior, we have the convergence in the distribution as $n\to \infty:$
   \begin{align*}
       &R_{1,2}|\mathcal{F}_1\to 0.
       \\& (\sigma_1,...,\sigma_p)|\mathcal{F}_p\to Y_p\text{ and } R_{1,2,p}|\mathcal{F}_p\to 0. 
   \end{align*}
   where $Y_p$ is a random variable and $R_{1,2,p}$ is the overlap between all but $1,...,p$ tuples.
\end{thm}
The above theorem implies this corollary which gives controlled probabilistic energy landscape.
\begin{cor}[Tuning probabilistic energy landscape]
    For any given temperature, by tuning the coefficients $\alpha(p)$, we can tune $p_t$ above to be our desired number. For any given sequence, we can make the dominant interaction come from the $t$-spin interaction with $p_t$. Furthermore, for any given sequence $a_t$ with $\sum a_t=1$ we can find $\alpha(p)'s$ such that this model shows $t$-RSB with probability $a_t$.
\end{cor}

\subsection{Sketch of proof} 

We reduce its model to non-intersecting monomials. We suggest the method called \emph{NIMR} (Non-intersecting monomial reduction). We start with the observation that we can directly compute the partition function of $H_n(\sigma)=H\sigma_1...\sigma_p n^{-(p-2)/2}$ and investigate the properties of this function. Surprisingly, this function can be computed exactly as a function of the moments of standard normal variable. This function exhibits a phase transition at a certain threshold $H_p^*$. The free energy scale shows the phase transition of the scale $n^{2-p}$ to the scale $n$ as $H$ changes from below the threshold to above the threshold. When we calculate this value, we use the Taylor expansion. This gives a polynomial-type expansion for the partition function, and it also exhibits a phase transition.  Under the threshold, the degree near the 0 part is enough to obtain its asymptotics; however, above the threshold, its essential part is concentrated near $\lambda_p(H)n$.
We extend this result to the \emph{NIM} model whose Hamiltonian consists of monomials that do not intersect using convexity and these concentration properties. This NIM corresponds to the largest $n^{\epsilon}$ interactions in our original model for $2\leq p\leq M\log n $. This will be proved using a random graph argument in the next section. In Section 4, we lift those NIM model results to our original heavy-tail model. We first divide our Hamiltonian into 5 parts. The first parts correspond to the NIM model, and we prove that the other parts are small enough. To show such a smallness for each part, we use the ideas of heavy-tail random matrix theory, coloring argument, and concentration inequality.

\subsection{Organization of the paper}
Section~2 introduces the heavy-tailed mixed $p$-spin model, fixes notation, and analyzes a single monomial Hamiltonian as a warm-up.
Section~3 establishes the full phase diagram and Gibbs geometry for the
\emph{Non-Intersecting Monomial (NIM) model}.   
Section~4 lifts those results to the complete Hamiltonian and proves our main theorems. Appendix A contains calculations about single monomial analysis. Appendix B calculates the refined free energy fluctuation for the subthreshold regime with $\alpha<1$.

\section{Model description and single monomial analysis}
In this section, we explain the model and Hamiltonian normalization. Moreover, we explain the lemmas about phase transition for calculating $\E[\exp(H\sigma_1...\sigma_pn^{-(p-2)/2})]$ and understanding its term by term of Taylor expansion. This phase transition and convexity of that term by term analysis will be used to prove concentration property and free energy phase transition for NIM model.
\begin{defn}[Heavy-tailed random variable] \label{stable law}
  We say that a random variable $X$ is heavy-tailed with exponent $\alpha$ if $\P(|X|>u)=L(u){u^{-\alpha}}$ for some $\alpha \in (0, 2)$ and some slowly varying function $L$, i.e. $\lim_{x\to\infty} L(tx)/L(x)=1$ for all $t$. We further assume that for any $\delta > 0$, there exists $x_0$ such that $L(x) e^{x^{\delta}}$ is an increasing function on $(x_0, \infty)$. 
\end{defn}
\begin{rem}\label{rem:slow_vary}
In Definition \ref{stable law}, the assumption on the monotonicity of the function $L(x) e^{x^{\delta}}$ is purely technical, this is just used to ensure that the largest terms and the second largest terms are not too close as  \cite{kim2024fluctuations}. This assumption is satisfied by many slowly varying functions considered in the literature, especially poly-log functions $(\log(x))^p$ for some $p \in \mathbb{R}$.
\end{rem}
We also define that an event occurs with high probability to describe its behaviour mathematically rigorous.
\begin{defn}[High probability event]
  We say that an event $F_n$ holds with high probability if there exist (small) $\delta>0$ and (large) $n_0$ such that $\P(F_n)\geq 1-n^{-\delta}$ holds for any $n > n_0$.
\end{defn}
We begin by computing its monomial expansion with respect to a standard normal random variable. The proofs of Lemmas 2.4–2.7 and 2.10–2.14 are deferred to Appendix \ref{app:proof}.
\begin{lem}\label{integral}
Let $\sigma=(\sigma_1,\dots,\sigma_n)$ be a point on the sphere $\{\sigma\in\mathbb{R}^n: \sum_{i=1}^n\sigma_i^2=n\}$ endowed with the uniform probability measure. Then for any nonnegative integers $i_1,\dots,i_k$, we have
\[
\mathbb{E}\bigl[\sigma_1^{\,i_1}\cdots \sigma_k^{\,i_k}\bigr]=n^{\frac{1}{2}\sum_{t=1}^k i_t}\frac{\mathbb{E}[|X|^{\,n-1}]\mathbb{E}[X^{\,i_1}]\cdots \mathbb{E}[X^{\,i_k}]}{\mathbb{E}\bigl[|X|^{\,i_1+\cdots+i_k+n-1}\bigr]},
\]
and \[
\mathbb{E}\bigl[|\sigma_1|^{\,i_1}\cdots |\sigma_k|^{\,i_k}\bigr]=n^{\frac{1}{2}\sum_{t=1}^k i_t}\frac{\mathbb{E}[|X|^{\,n-1}]\mathbb{E}[|X|^{\,i_1}]\cdots \mathbb{E}[|X|^{\,i_k}]}{\mathbb{E}\bigl[|X|^{\,i_1+\cdots+i_k+n-1}\bigr]},
\]
where $X\sim N(0,1)$.
\end{lem}

Applying the above lemma, we can compute
$\E[\exp(H \sigma_1\cdots\sigma_pn^{-(p-2)/2})]$. Remarkably, this quantity exhibits a phase transition at the threshold
\(H_p^{*}\); the value itself, as well as the qualitative behaviour of its
Taylor expansion, changes abruptly when \(H\) crosses \(H_p^{*}\).
\begin{lem}\label{exp_basic} The expectation value of the exponential function exhibits a phase transition as described below. There is a threshold $H_p$ such that for $|H|>H^*_p+n^{-\epsilon}$,
\[
    \log \E[\exp(H\sigma_1\cdots\sigma_pn^{-(p-2)/2})]= nf_p(H)(1+O(n^{-\epsilon}))
\]
and  
for $H<H^*_p-n^{-\epsilon}$,
\begin{itemize}
    \item $p\geq 3$:\[
    \log \E[\exp(H\sigma_1\cdots\sigma_pn^{-(p-2)/2})]= \frac{1}{2}n^{2-p}H^2(1+O(n^{-\epsilon})).\]
    \item $p=2$:
    \[
    \log \E[\exp(H\sigma_1\cdots\sigma_pn^{-(p-2)/2})]= (1-H^2)^{-1/2}(1+O(n^{-\epsilon})).\]

\end{itemize}

\end{lem}

Also, we show that the terms in Taylor expansion are concentrated in a small region.
\begin{lem}\label{phase3}For $p \ge 3$ and
\(
\widehat{H}= H\sigma_1...\sigma_p n^{-\tfrac{p-2}{2}},
\)
we have the following phase transition.
There is a threshold $H_p^*$ such that:
\begin{itemize}
\item If $|H| < H_p^*-n^{-\epsilon}$, then
\[
\mathbb{E}\bigl[\exp(\widehat{H})\bigr]
=
1+\frac{1}{2}n^{2-p}H^{2}
\bigl[1 + O(n^{-\epsilon})\bigr].
\]
\item If $|H| > H_p^*+n^{-\epsilon}$, then
\[
\mathbb{E}\bigl[\exp(\widehat{H})\bigr]
=
\sum_{\ell =\lambda_p(H)\,n(1-n^{-\epsilon})}^{\ell =\lambda_p(H)\,n(1+n^{-\epsilon})}
\frac{\mathbb{E}\bigl[\widehat{H}^{2\ell}\bigr]}{(2\ell)!}
\;\bigl[\,1 + O(n^{-\epsilon})\bigr].
\]
\end{itemize}
\end{lem}
For the case of $p=2$, we have a slightly different phase transition.
\begin{lem}\label{phase2}
For $p=2$, the behavior is different and we see another phase transition:
\begin{itemize}
\item For $|H| < 1$, 
\[
\mathbb{E}\bigl[\exp(\widehat{H})\bigr]
\;=\;
\sum_{\ell=0}^{n^\epsilon}
\frac{\mathbb{E}\bigl[\widehat{H}^{2\ell}\bigr]}{(2\ell)!}\,
\;\Bigl[\,1 + O(n^{-\epsilon})\Bigr].
\]
\item For $|H| > 1+n^{-\epsilon}$,
\[
\mathbb{E}\bigl[\exp(\widehat{H})\bigr]
\;=\;
\sum_{\tfrac{H-1}{4}n(1-n^{-\epsilon})}^{\tfrac{H-1}{4}n(1+n^{-\epsilon})}
\;\frac{\mathbb{E}\bigl[\widehat{H}^{2\ell}\bigr]}{(2\ell)!}
\;\Bigl[\,1 + O(n^{-\epsilon})\Bigr].
\]
\end{itemize}
\end{lem}
We define $\lambda_p,H_p^*,f_p$ from the above lemmas.
\begin{defn}\label{def:threshold}
    We define $\lambda_p(H),H_p^*$ as the concentration location of the Taylor expansion and the expectation phase transition threshold as the lemmas above. Especially $\lambda_2(H)=\frac{H-1}{4}$ and $H_2^*=1$. For $p\geq 3$,
    $\lambda_p(H)$ is defined when $|H|>\frac{p^{p-1}}{2(p-2)^{(p-2)/2}}$ and the larger solution of the equation
    \[
        2\log(H)+(p-2)\log(2\lambda_p(H))-p\log(2p\lambda_p(H)+1)=0.
    \]
    $H_p^*$ is the solution of the equation
    \begin{align*}
    g(p,\lambda_p(H),H)=0
\end{align*}
where 
\[
    g(p,c,H)=2c\log( H)-2c\log(2c)+2c+pc\log 2c -\frac{2pc+1}{2}\log(2pc+1).
\]
Furthermore, for all $p\geq 2$ and $|H|>H_p^*$,
\[
    f_p(H)= g(p,\lambda_p(H),H).
\]
We also have \[
f_p(H)=2\lambda_p(H)-\frac{1}{2}\log (1+2p \lambda_p(H)).
\]
For $|H|<H_p^*$, we define $f_p(H)=0.$
\end{defn}
\begin{rem}
    Using simple calculation, we can check that $H_p^*>\frac{p^{p-1}}{2(p-2)^{(p-2)/2}}$.
\end{rem}
We need to analyze this function more thoroughly to understand its behavior.
Also, for the above threshold case of overlap calculation, we need odd power term expansion. 
\begin{lem}
    For the above threshold case ($|H|>H^*_p$), odd power sum and even power sum are almost the same, and also the odd power expansions are also concentrated.
    \[\sum_{\,\ell \;=\;\lambda_p(H)\,n(1-n^{-\epsilon})}^{\,\ell \;=\;\lambda_p(H)\,n(1+n^{-\epsilon})}
\frac{\mathbb{E}\bigl[|\widehat{H}|^{2\ell+1}\bigr]}{(2\ell+1)!}
\;=\sum_{\,\ell \;=\;\lambda_p(H)\,n(1-n^{-\epsilon})}^{\,\ell \;=\;\lambda_p(H)\,n(1+n^{-\epsilon})}
\frac{\mathbb{E}\bigl[(\widehat{H})^{2\ell}\bigr]}{(2\ell)!}
\;\bigl[\,1 + O(n^{-\epsilon})\bigr],\]
and
\[
    \sum_{\ell=0}^{\infty}
\frac{\mathbb{E}\bigl[|\widehat{H}|^{2\ell+1}\bigr]}{(2\ell+1)!}
\;=\sum_{\,\ell \;=\;\lambda_p(H)\,n(1-n^{-\epsilon})}^{\,\ell \;=\;\lambda_p(H)\,n(1+n^{-\epsilon})}
\frac{\mathbb{E}\bigl[|\widehat{H}|^{2\ell+1}\bigr]}{(2\ell+1)!}
\;\bigl[\,1 + O(n^{-\epsilon})\bigr].
\]
\end{lem}
\begin{lem}\label{convexity1}
The function
    \[
    g(z)= z\log Hn-\log\Gamma(z+1)+p\log\mathbb{E}[|X|^{z}]+\log \E[|X|^{n-1}]
    \]
    is convex and
    \[
    g(z)=\log \frac{\E[|\widehat{H}|^z]\E[|X|^{n-1+pz}]}{z!} \text{ for }z\in \N_0.
    \]
\end{lem}
This also ensures convexity for its linear combination as in the following.
\begin{lem}\label{convexity2}
    The function 
    \[
        g(z_1,...,z_k)= \sum_{i=1}^k g(a_iz_i)
    \]
    satisfies convexity.
\end{lem}
The above lemma will be used to prove the concentration property to obtain the formula of overlap.
Moreover, we have a simple upper bound for the non-intersecting monomial Hamiltonian.
\begin{lem}\label{decom_1}
    For mutually exclusive $I_t$'s($I_i\cap I_j=\emptyset$ for every pair.), 
    \[
        \E[\sigma_{I_1}^{a_1}\cdots\sigma_{I_t}^{a_t}]\leq \E[\sigma_{I_1}^{a_1}]\cdots\E[\sigma_{I_t}^{a_t}]
    \]
    where $\sigma_I= \prod_{i\in I}\sigma_i$ and $I_i=\{b_{i1},...,b_{ic(i)}\}$.
\end{lem}
\begin{lem}\label{decompose}
    For $\sigma_I$, if each pair of $I, J\in \mathcal{I}$ is nonintersect, then the inequality below holds \[
        \E[\exp(\sum_{I\in \mathcal{I}} H_I\sigma_I)]\leq \prod_{I\in \mathcal{I}} \E[\exp[H_I\sigma_I]]
    \]
\end{lem}

\section{Energy landscape of Non-Intersecting Monomial (NIM) model}
In this section, we investigate the energy landscape of the \emph{NIM} model which has $n^{\epsilon_0}$ interactions for each $2\leq p\leq M\log n$ and each of the corresponding monomials are not intersecting.
\begin{defn}
    NIM model is the mixed $p$-spin spherical model with Hamiltonian
    \[
    H_n(\sigma)= \sum_{2\leq p\leq M\log n} H_{n,p},  \quad H_{n,p}= \sum_{i<n^{\epsilon_0}} H_{i,p}\sigma_{i1}...\sigma_{ip}n^{-(p-2)/2} 
    \] where $H_{i,p}$ are values and $\sigma_{ik}$ are non-intersecting spins in $\{\sigma_1,...,\sigma_n\}$.
\end{defn}
In this section, we deal with their energy landscape case by case. 
For their coefficient scale, we deal with three possible regimes. First one is no dominance (\emph{ND}) regime which is all interactions are under the threshold. Second one is multiple comparable dominance (\emph{MD}) which is few multiple dominant terms are comparable as poly($n$) scale difference and others are exponentially small $\exp(n^{\epsilon
})$. Remaining part is single dominant regime (\emph{SD}) which is single dominant term is exponentially larger than others. In Section 3.1, we deal with ND case. In Section 3.2, we deal with free energy and its concentration property for MD regime. In Section 3.3 we deal with energy landscape for SD regime. We only require SD and ND for heavy tail extension. 
\subsection{All under the threshold}
In this subsection, we prove the following propositions about the free energy and energy landscape of the NIM model when every interaction is below the threshold.
For convenience, we first rearrange the spin variables in a more manageable manner.
Let 
\[
  K_i:\text{ (2-spin coefficient)}, 
  \quad 
  H_i:\text{ ($l_i$-spin coefficient with }l_i \ge 3).
\]
More precisely, we define its Hamiltonian as\[
H_n(\sigma)= \sum_{i=1}^k \widehat{K}_i+\sum_{i=1}^h \widehat{H}_i
\] where
\[
  \widehat{K}_i = K_i\sigma_{i_1}\sigma_{i_2},
  \quad\text{and}\quad
  \widehat{H}_i=H_i\sigma_{i_1}\sigma_{i_2}\cdots\sigma_{i_{l_i}}
  n^{-\tfrac{l_i-2}{2}}
\]
which of each $\sigma_{i_j}$ are mutually different spin variables in $\{\sigma_1,...,\sigma_n\}$.
\begin{prop}[Free energy of NIM model]\label{under_free}
    For the NIM model, if all the interactions are below the threshold $H_p^*$ in Definition \ref{def:threshold}, we have below the sampled free energy. Let $\beta(p)$ be the smallest existing p-interaction.
    If $\beta(p)=2$,
    \[
        \log Z_n= \log \prod_{\widehat{H}_i:2-\text{spin}} (1-H^2_i)^{-1/2}(1+O(n^{-\epsilon})).
    \]
    Otherwise,
    \[
        \log Z_n=\frac{1}{2}n^{2-\beta(p)}(1+O(n^{-\epsilon}))\sum_{\widehat{H}_i:\beta(p)-\text{spin}} H_i^2.
    \]
\end{prop}
\begin{proof}
We first prove the $\beta(p)\geq3$ case. In this case, the Hamiltonian only includes $\widehat{H}_i$ parts and the partition function we are aiming to obtain is
\(
    \E[\exp(\sum_i \widehat{H}_i)].
\)
This has an upper bound
\[
    Z_n=\E[\exp(\sum_i \widehat{H}_i)]\leq \prod_i\exp(\widehat{H}_i)=1+\frac{1}{2}(1+O(n^{-\epsilon}))\sum_{p\geq 3}n^{2-p}\sum_{\widehat{H}_i:p\text{-spin}} H_i^2
\]
due to Lemma \ref{decompose}.
\\By the inequality $\exp(x)\geq 1+x+\frac{1}{2}x^2+\frac{1}{6}x^3$, we also have the lower bound
\[
    \E[\exp(H_n(\sigma))]\geq 1+\E[H_n(\sigma)]+\frac{1}{2}\E[H_n(\sigma)^2]+\frac{1}{6}\E[H_n(\sigma)^3].
\]
Since monomials are not intersecting to each other, we have
\[
    \E[H_n(\sigma)]=\E[H^3_n(\sigma)]=0
\]
and this implies the right side of above inequality is just
\[
    1+\frac{1}{2}\E[H_n(\sigma)^2]=1+\frac{1}{2}\sum_{p\geq 3}n^{2-p}\sum_{\widehat{H}_i:p\text{-spin}} H_i^2.
\]
This proves that for $\beta(p)\geq 3$ case, we have
 \[
        Z_n=1+\frac{1}{2}n^{2-\beta(p)}(1+O(n^{-\epsilon}))\sum_{i:\beta(p)-\text{spin}} H_i^2
    \]
and
 \[
       \log  Z_n=\frac{1}{2}n^{2-\beta(p)}(1+O(n^{-\epsilon}))\sum_{i:\beta(p)-\text{spin}} H_i^2.
    \]
We consider $\beta(p)=2$ case and the partition function is
\[
  Z_n=\mathbb{E}\!\Bigl[\exp\Bigl(\sum_{i}\widehat{K}_i \;+\; \sum_{i}\widehat{H}_i\Bigr)\Bigr].
\]
Since their spin variables are non-intersect, Lemma \ref{decompose} gives the upper bound
\[
 Z_n
  \leq
  \mathbb{E}\!\Bigl[\exp\Bigl(\sum_{i}\widehat{K}_i\Bigr)\Bigr]
  \mathbb{E}\!\Bigl[\exp\Bigl(\sum_{i}\widehat{H}_i\Bigr)\Bigr].
\]
Since above proof for $\beta(p)\geq 3$ case implies 
\(\mathbb{E}\bigl[\exp(\sum_{i}\widehat{H}_i)\bigr]=1 + O(n^{-\epsilon})\), 
Lemma \ref{decompose} gives
\[
  Z_n
  \le\
  \mathbb{E}\!\Bigl[\exp\Bigl(\sum_{i}\widehat{K}_i\Bigr)\Bigr]
  \bigl(1+O(n^{-\epsilon})\bigr)\le\prod_{i} \mathbb{E}\bigl[\exp(\widehat{K}_i)\bigr](1+O(n^{-\epsilon})\bigr).
\]
Furthermore, we have an upper bound
\[
    \E[\exp(\widehat{K})]=\sum_{a=0}^\infty (Kn)^{2a}\frac{\E[X^{2a}]^2}{(2a)!}\frac{\E[|X|^{n-1}]}{\E[|X|^{n-1+4a}]}\leq \sum_{a=0}^{\infty}\frac{1}{2^{2a}}\binom{2a}{a}K^{2a}=(1-K^2)^{-1/2}
\]
for any $K<1$.
This gives the upper bound for the partition function \[
    Z_n\leq \prod_i(1-K_i^2)^{-1/2}(1+O(n^{-\epsilon})).
\]
The remaining part is to shows its lower bound and let us begin with this lower bound:
\[
  \mathbb{E}\!\Bigl[\exp\Bigl(\sum_{i}\widehat{K}_i + \sum_{i}\widehat{H}_i\Bigr)\Bigr]
  \;\;\ge\;\;
  \mathbb{E}\!\Bigl[\exp\Bigl(\sum_{i}\widehat{K}_i\Bigr)\Bigr].
\]
One also obtains lower bounds by expanding the exponential in power series;
\begin{align}\label{eq:con_un_1}
\mathbb{E}\Bigl[\exp\bigl(\sum_{i} \widehat{K}_i\bigr)\Bigr]
  &\geq \sum_{0\le a_i \le n^{\epsilon}\text{ for all }i}\frac{\E[|X^{n-1}|]}{\E[|X|^{n-1+4\sum_i a_i}]}\prod_{i} K_i^{2a_i}\frac{\E[X^{2a_i}]^2}{(2a_i)!}\notag
  \\&\geq
  (\frac{n}{n+ {4n^{\epsilon}l}})^{2n^{\epsilon}l}
  \sum_{0\leq a_i\leq n^{\epsilon}\text{ for all }i} 
  \prod_{i} K_i^{2a_i}\frac{\E[X^{2a_i}]^2}{(2a_i)!}\notag
  \\ &\geq O(\exp(-8n^{2\epsilon-1}l^2))\prod_i\sum_{0\leq a_i\leq n^{\epsilon}\text{ for all }i}\frac{1}{2^{2a_i}}\binom{2a_i}{a_i}K_i^{2a_i}.
\end{align}
The first inequality above comes from
\[
    \frac{\E[|X^{n-1}|]}{\E[|X|^{n-1+4\sum_i a_i}]}\geq (\frac{n}{n+ {4n^{\epsilon}l}})^{2n^{\epsilon}l}.
\]
We also have
\[
    \sum_{a_i>n^{\epsilon}}\frac{1}{2^{2a_i}}\binom{2a_i}{a_i}K_i^{2a_i}\leq \frac{K_i^{2n^{\epsilon}}}{1-K_i^2}
\]
and
\[
    \sum_{a_i=0}^{\infty}\frac{1}{2^{2a_i}}\binom{2a_i}{a_i}K_i^{2a_i}=(1-K_i^2)^{-1/2}.
\]
This implies more precise lower bound:
\begin{align}\label{eq:con_un_2}
\mathbb{E}\Bigl[\exp\bigl(\!\sum_{i}\widehat{K}_i\bigr)\Bigr]
  \;&\ge\;
  O\bigl(\exp(-8\,n^{2\epsilon-1}\,l^2)\bigr)
  \;\prod_{i}\,((1-K_i^2)^{-1/2}- \frac{K_i^{2n^{\epsilon}}}{1-K_i^2})\notag
  \\&\ge(1+O(n^{-\epsilon}))\prod_{i}(1-K_i^2)^{-1/2}.
\end{align}
This finally gives the lower bound 
\[
    Z_n\geq \prod_i(1-K_i^2)^{-1/2}(1+O(n^{-\epsilon})).
\]
Combining all above and we finally have 
\begin{align*}
\mathbb{E}\!\Bigl[\exp\Bigl(\sum_i \widehat{K}_i + \sum_i \widehat{H}_i\Bigr)\Bigr]
  =   \prod_i\,(1-K_i^2)^{-1/2}   \;\bigl(1 + O(n^{-\epsilon})\bigr).
\end{align*}
\end{proof}
With the above proposition and \eqref{eq:con_un_1}, \eqref{eq:con_un_2}, we also have the following concentration property:
\beq\label{eq:con}
    Z_n= \sum_{(a,b)\in A} \prod_i \frac{\widehat{K}^{2a_i}}{(2a_i)!} \prod_i \frac{\widehat{H}^{2b_i}}{(2b_i)!}[1+O(n^{-\epsilon})]
\eeq
where $a=(a_1,...,a_k)$ and $b=(b_1,...,b_h)$, $A=\{(a,b)| 0\leq a_i\leq n^{\epsilon}\text{ for all } i ,b=0\}$. This equality will be used to prove the following proposition about the Gibbs measure structure. Their overlap converges to 0 by showing their variance is small enough as below. 
\begin{prop}\label{under_overlap}
    For under the threshold, its overlap converges to 0. 
    More precisely, its variance over Gibbs measure satisfies
    \[
        \langle R_{1,2}^2\rangle=O(n^{-\epsilon}).
    \]
\end{prop}
We use the same notation for its Hamiltonian as above and $'$ means the replica of the spin. More specifically, we have two Hamiltonians 
\[
    H_n(\sigma)=\sum_{i=1}^k \widehat{K}_i + \sum_{i=1}^h\widehat{H}_i, \quad H_n(\sigma')=\sum_{i=1}^k \widehat{K}'_i+\sum_{i=1}^h\widehat{H}'_i.
\]
We are aiming to figure out the expectation value of their overlap \[ \mathbb{E}\!\Bigl[R_{1,2}^2       \,\exp\Bigl(H_n(\sigma)+H_n(\sigma')\Bigr)\Bigr]. \]
Here the expectation means expectation over two independent uniform measure on the sphere.
This can be expanded by Taylor series as
\begin{align*}
&\mathbb{E}\!\Bigl[(\frac{\sigma_1\sigma_1'+\cdots+ \sigma_n\sigma_n'}{n})^2       \,\exp\Bigl(\sum_{i=1}^k \widehat{K}_i + \sum_{i=1}^h\widehat{H}_i+\sum_{i=1}^k \widehat{K}'_i+\sum_{i=1}^h\widehat{H}'_i\Bigr)\Bigr]
\\ & \; = \mathbb{E}\!\Bigl[(\frac{\sigma_1\sigma_1'+\cdots +\sigma_n\sigma_n'}{n})^2       \,\bigl(\sum_{a,b,c,d}\prod_i \frac{\widehat{K}_i^{a_i}}{a_i!}\prod_i \frac{\widehat{H}_i^{b_i}}{b_i!}
\prod_i \frac{\widehat{K}_i'^{c_i}}{c_i!}
\prod_i \frac{\widehat{H}_i'^{d_i}}{d_i!}\bigr)\Bigr]
\end{align*}
where $a=(a_1,...,a_k),b=(b_1,...,b_h),c=(c_1,...,c_k),d=(d_1,...,d_h)$. We further define \[h(a,b,c,d)=\prod_i \frac{\widehat{K}_i^{a_i}}{a_i!}\prod_i \frac{\widehat{H}_i^{b_i}}{b_i!}
\prod_i \frac{\widehat{K}_i'^{c_i}}{c_i!}
\prod_i \frac{\widehat{H}_i'^{d_i}}{d_i!}\] to elaborate more convenient way.
This sum is expressed as the linear combination of the form
\[
\E[\sigma_i\sigma_i'\sigma_j \sigma_j' h(a,b,c,d)\bigr].
\]
The only possible conditions that can make the above expectation non-zero are
$i=j$ and all $a_i,b_i,c_i,d_i$ are even or $\widehat{K}_{k}=K_{k}\sigma_i\sigma_j$ for some $k$ and $a_k$ is odd, and other terms are even.
For the latter case, we have the upper bound using their nonintersecting property. This shows
\begin{align*}
&\E[\sigma_i\sigma_i'\sigma_j \sigma_j' \exp\Bigl(\sum_i \widehat{K}_i + \sum_i \widehat{H}_i+\sum_i \widehat{K}'_i+\sum_i \widehat{H}'_i\Bigr)\bigr]\\ &\; \leq \E[\sigma_i\sigma_j \exp(\widehat{K}_k)]^2\prod_{l\neq k}\E[\exp(\widehat{K}_l)]^2\prod_l \E[\exp(\widehat{H}_l)]^2.
\end{align*}
The first term has an upper bound using expansion as
\begin{align*}
    \E[\sigma_i\sigma_j \exp(\widehat{K}_k)]&= \sum_{a=0}^{\infty}\frac{\E[X^{2a+2}]^2(K_kn)^{2a+1}\E[|X|^{n-1}]}{(2a+1)!\E[|X|^{n-1+4a+2}]}
    \; \leq\sum_{a=0}^{\infty}\frac{\E[X^{2a+2}]^2}{(2a+1)!}K^{2a+1}_k \\
    &\; \leq \sum_{a=0}^{\infty}(2a+2)K_k^{2a+1}  \; \leq \frac{2K_k}{(1-K^2_k)^2}  \; \leq \E[\exp(\widehat{K}_k)]\frac{2K_k}{(1-K_k)^{3/2}}.
\end{align*}
Let $B'=\{(i,j)|\widehat{K}_a=K_a\sigma_i\sigma_j\text{ for some }a\}$ and $B=\{a|\widehat{K}_a=K_a\sigma_i\sigma_j \text{ for some } i,j\}$, then
\begin{align*}
    &\sum_{(i,j)\in B'} \E[\sigma_i\sigma_j \sigma_i'\sigma_j' \exp(\sum_l \widehat{K}_l +\sum_l \widehat{H}_l + \sum_l\widehat{K}_l' +\sum_l \widehat{H}'_l)]
    \\&\; \leq \sum_{(i,j)\in B'}\E[\sigma_i\sigma_j \exp(\widehat{K}_k)]^2\prod_{l\neq k}\E[\exp(\widehat{K}_l)]^2\prod_l \E[\exp (\widehat{H}_l)]^2
    \\ & \; \leq \E[\exp (\sum_i \widehat{K}_i+\sum_i \widehat{H}_i)]^2(1+O(n^{-\epsilon}))\sum_{k\in B} (\frac{2K_k}{(1-K_k^2)^{3/2}})^2.
\end{align*}
The remaining part is to calculate all powers even. We now only care about that parts by omitting others. We define 
\(
A =\{\, (a, b) \;|\; \;0 \le a_i \le n^\epsilon \text{ for all } i, \ b=0\}.
\)
Using the bound 
\(\bigl|R_{1,2}\bigr|\le 1\) and \eqref{eq:con}, we have the upper bound
\begin{align*}
\sum_{(a,b)\notin A} \mathbb{E}[R_{1,2}^2h(2a,2b,2c,2d)] \leq \sum_{(a,b)\notin A} \mathbb{E}[h(2a,2b,2c,2d)] 
  \leq O(n^{-\epsilon})\E[\exp (
H_n(\sigma)
)]^2.
\end{align*}
Similar argument implies
\[
\sum_{(c,d)\notin A}  \mathbb{E}[R_{1,2}^2h(2a,2b,2c,2d)] 
\leq O(n^{-\epsilon})\E[\exp (
\sum_i \widehat{K}_i
)]^2.
\]
\\The remaining part satisfies
\begin{align*}
\sum _{(a,b)\in A, (c,d)\in A}\mathbb{E}\!\Bigl[R_{1,2}^2       \,h(2a,2b,2c,2d)\Bigr]
 &  = \sum _{(a,b)\in A, (c,d)\in A}\mathbb{E}\!\Bigl[\frac{\sum _{i=1}^n\sigma _i^2 \sigma_i'^2}{n^2}       \,h(2a,2b,2c,2d)\Bigr]
\\ &  \leq \frac{(2n^{\epsilon}+1)^2}{n^2}\E[\exp(H_n(\sigma))]^2.
\end{align*}
Summing everything yields
\begin{align*}
\E[R_{1,2}^2\exp(H_n(\sigma)+H_n(\sigma'))]
 \ \leq \E[\exp(H_n(\sigma))]^2(\frac{4}{n^2}\sum_{i\in B}(\frac{K_i}{(1-K_i^2)^{3/2}})^2+O(n^{-\epsilon}+\frac{(2n^{\epsilon}+1)^2}{n^2}) ).
\end{align*}
Consequently, we have 
\[
    \langle R_{1,2}^2\rangle=\frac{\E[R_{1,2}^2\exp(H_n(\sigma)+H_n(\sigma'))]}{\E[\exp(H_n(\sigma))]^2}= O(n^{-\epsilon})
\]
indicating that the limiting distribution approaches 0. This confirms the presence of replica symmetry. This trivial overlap also implies ultrametricity and temperature chaos for this region. Now we deal with surpassing the threshold case.

\subsection{Above threshold: Free energy and concentration property}
When NIM model has above threshold interactions, we define its Hamiltonian as follows: 
\[
H_n(\sigma)=
\sum_{i=1}^h \widehat{H}_i
\;+\;
\widehat{K}
\;+\;
\sum_{i=1}^l \widehat{L}_i
\;+\;
\sum_{i=1}^m \widehat{M}_i,
\]
where $\widehat{H}_i$ are $h_i$-spin terms, $\widehat{K}$ is a 2-spin term, $\widehat{L}_i$ are $l_i$-spin terms, and $\widehat{M}_i$ still other 2-spin terms.  Each terms have of the form $\widehat{X}=X\sigma_I n^{-(|I|-2)/2}$ where X is real number and $\sigma_I$ is the corresponding spin monomial. Here we use notation $\sigma_I=\prod_{i\in I}\sigma_i$.
Moreover, these terms have these ordering of the size below
\[
\mathbb{E}\bigl[\exp(\widehat{H}_i)\bigr]
\;\sim\;
\mathbb{E}\bigl[\exp(\widehat{K})\bigr]
\;>\;
\exp(n^{1-2\epsilon})\,\mathbb{E}\bigl[\exp(\widehat{L}_i)\bigr],
\quad
\exp(n^{1-2\epsilon})\,\mathbb{E}\bigl[\exp(\widehat{M}_i)\bigr],
\]
for some $\epsilon_0 < \epsilon<1/2$.
We write $A\sim B$ if $n^{-M} A \le B \le n^M A$ for some $M>0$.
\\We define tuples
\[
a \;=\;(a_1,\dots,a_h), 
\quad
b,
\quad
c \;=\;(c_1,\dots,c_l),
\quad
d \;=\;(d_1,\dots,d_m),
\]
and consider
\[
f(a,b,c,d)
\;=\;
\mathbb{E}\Biggl[
\prod_{i=1}^h \frac{\widehat{H}_i^{2\,a_i}}{(2\,a_i)!}
\;\cdot\;
\frac{\widehat{K}^{2\,b}}{(2\,b)!}
\;\cdot\;
\prod_{i=1}^l \frac{\widehat{L}_i^{2\,c_i}}{(2\,c_i)!}
\;\cdot\;
\prod_{i=1}^m \frac{\widehat{M}_i^{2\,d_i}}{(2\,d_i)!}
\Biggr].
\]
Also, we define $T$ as the maximum possible value of this function.
\\Our first goal in this section is to show that this sum is concentrated on these regions:
\[
A_i = \bigl\{
\lambda_{i}(1-n^{-\epsilon})\,n \le a_i \le \lambda_{i}(1+n^{-\epsilon})\,n,\
0 \le b \le n^\epsilon,\; c_j=0,\; d_j \le n^{\epsilon}
\text{ for all } j\bigr\}
\]
and
\[
B = \bigl\{
\lambda(1-n^{-\epsilon})\,n\le b \le \lambda(1+n^{-\epsilon})\,n,\;
a_j=0,\; c_j=0,\; d_j \le n^{\epsilon} \text{ for all } j
\bigr\}
\]
where $\lambda_i=\lambda_{h_i}(H_i)$ 
and  $\lambda= \lambda_2(K)$.
\\From now on, we express $e=(a,b,c,d)$ and $A=\cup A_i\cup B$. We will prove that the sum of function is concentrated on small subsets as proposition below.
\begin{prop}\label{conc_nim}
The Taylor expansion summation of its partition function is concentrated on its subset of tuples:
    \[
  \sum_{e} f(e) 
  \;=\; 
  \sum_{e \,\in\, \cup A_i \cup B} f(e)\,\bigl(1 + O\bigl(n^{-\epsilon}\bigr)\bigr).
\]
\end{prop}
\begin{proof}
The main idea is using convexity about the function $f(e)$.
We define their sum with weight the number of corresponding spin as $S=\sum_i a_ih_i+2b+\sum_i l_ic_i+\sum_i 2d_i$.
For fixed constraint $S$, using the convexity in Lemma \ref{convexity2}, we obtain the following upper bound:
\[
  f(a,b,c,d)
  \;\le\;
  \max\Bigl\{
    f\bigl(a_i = S/h_i\bigr),\,
    f\bigl(b = S/2\bigr),\,
    f\bigl(c_i = S/l_i\bigr),\,
    f\bigl(d_i = S/2\bigr)
  \Bigr\}.
\]
Our proof strategy is to reduce the set step by step. Before we calculate further,
We define the operation
\[
  f_{x_i}(e)
  \;=\;
  \text{(the value of $f$ after incrementing $x_i$ in $e$ by 1)},
  \quad
  x\in\{a,b,c,d\}.
\]
Then we have the following inequalities.
For $x_i = b$ or $d_i$, we have
\beq\label{eq:op_bd}
  \frac{f_{x_i}(e)}{f(e)}
  \;\le\;
  \frac{(X_i n)^2}{(n + 2S)(n + 2S + 2)}.
\eeq
  For $x_i = a_i$ or $c_i$, we have
\beq\label{eq:op_ac}
    \frac{f_{x_i}(e)}{f(e)}\leq \frac{(Hn)^2(2x_i+1)^h}{(2x_i+1)(2x_i+2)(n+2S)(n+2S+2)...(n+2S+2(h-1))}
\eeq
where $H$ is the coefficient and $h$ is the degree of the corresponding monomial. We first reduce the case to $\cup A_i' \cup B'$ where \[A_i'=\{\lambda_ih_in(1-n^{-\epsilon})\leq S\leq \lambda_ih_in(1+n^{-\epsilon})\}, \; B'=\{2\lambda n(1-n^{-\epsilon})\leq S\leq 2\lambda n(1+n^{-\epsilon})\}.\]
For \(e \notin \cup A_i' \cup B'\), \eqref{eq:op_ac} and \eqref{eq:op_bd} give
\[
  f(e) \le \exp\bigl(-C\,n^{\,1 - 2\epsilon}\bigr)\;T.
\]
We apply this inequality and we have 
\beq\label{eq:above_1}
\sum _{S\leq S', e\notin \cup A_i'\cup B'} f(e) \leq \exp\bigl(-C\,n^{\,1 - 2\epsilon}\bigr)\;T\binom{S'+Mn^{\epsilon_0\log n}}{Mn^{\epsilon_0}\log n }
\eeq
where $S'=\max\bigl\{\,\lambda_i h_in\bigl(1 + n^{-\epsilon}\bigr),\;
2\lambda n\bigl(1 + n^{-\epsilon}\bigr)\bigr\}$ and the last term is the upper bound of the number of pairs whose weighted sum $S$ is less than equal to $S'$. The number of variables is $Mn^{\epsilon_0}\log n $ from our NIM model definition.
The remaining part is the cases of $S>S'$. Summing over all $S > S'$, \eqref{eq:op_bd} and \eqref{eq:op_ac} give 
\begin{align}\label{eq:above_2}
  \sum_{S > S'} f(e)
  \;&\le\;
  \exp\bigl(-C\,n^{\,1 - 2\epsilon}\bigr)\;T 
  \;\sum_{S > S'} \binom{S+Mn^{\epsilon_0\log n}}{Mn^{\epsilon_0}\log n }(n^{-\epsilon})^{\tfrac{S-S'}{2}}\notag
  \\&\le\;\exp\bigl(-C\,n^{\,1 - 2\epsilon}\bigr)\;T 
  \;\binom{S'+Mn^{\epsilon_0\log n}}{Mn^{\epsilon_0}\log n }C'\notag
  \\&\le \; C'T\exp\bigl( -Cn^{1-2\epsilon}+n^{\epsilon_0}(\log n )^3\bigr).
\end{align}
\eqref{eq:above_1} and \eqref{eq:above_2} reduce our analysis to the case 
\(
  e \;\in\; \cup A_i' \cup B'.
\)
\\We next reduce to a smaller set of variables. Under the constraint 
\[
  \lambda_i n\bigl(1 - n^{-\epsilon}\bigr)<S<\lambda_i n\bigl(1 + n^{-\epsilon}\bigr)
\]
and $S_d=\sum_i 2d_i > 2\,n^{\epsilon'}$, we have
\[
  f(e)\;\le\;C^{\,2n^{\epsilon'}}\,T
  = \exp\bigl(\log(C)\,2n^{\epsilon'}\bigr)\;T.
\]
The number of pairs in this regime is bounded above as
\[
  \binom{S'+Mn^{\epsilon_0}\log n }{Mn^{\epsilon_0}\log n }<\exp(Mn^{\epsilon_0}(\log n )^3)
\]
and by choosing $\epsilon'> \epsilon_0$, we have
\[\sum_{e \in \{S_d > 2n^{\epsilon'}\}\cap (\cup_i A_i'\cup B)} f(e)
  <
  T\exp\bigl(-n^{\epsilon''}\bigr),
\]
for some $\epsilon'' > 0$.

This shows that the main contribution of partition function  is concentrated in
\[
  (\cup A_i' \cup B')
  \cap
  \{S_d < 2n^{\epsilon'}\}.
\]
In the region \(A_i' \cap \{S_d < 2n^{\epsilon'}\}\),  \eqref{eq:op_bd} and \eqref{eq:op_ac} give the upper bound
\[
   f(e) \le C^{2n^{\epsilon'}}T 
   =\exp(\log(C) 2\,n^{\epsilon'})T
\]
for \(b > n^{\epsilon'}\).
This implies the upper bound
\[
   \sum_{\,e \in A_i' \cap \{\,b > n^{\epsilon'}\}} f(e)
   < T\,\exp\bigl(-n^{\epsilon''}\bigr).
\]
This reduce each \(A_i' \cap \{S_d < 2n^{\epsilon'}\}\) to
\[
   A_i'' \;=\; 
   A_i' 
   \;\cap\;\{\,S_i < 2n^{\epsilon'},\; b<n^{\epsilon'}\}.
\]
Let us define
\(
   S_i = \sum_{j\neq i} a_jh_j + \sum_{j} c_jl_j
\)
and we choose small $\epsilon_2>0$ such that $\epsilon_0+\epsilon_2<\epsilon$. For \(S_i < n^{\epsilon_2}\), we have
\[
   f(e) <T \bigl(n^{\epsilon_2-1}\bigr)^{\sum_{j\neq i} a_j(h_j-2) + \sum_{j} c_j(l_j-2) }<Tn^{(\epsilon_2-1)(\sum_{j\neq i} a_j+\sum_j c_j)}.
\]
This gives the following upper bound 
\begin{align*}
   \sum_{1\leq S_i \leq n^{\epsilon_2},S\text{ : fixed}} f(e)&<\sum_{1\leq\sum_{j\neq i} a_j+\sum_j c_j }Tn^{(\epsilon_2-1)(\sum_{j\neq i} a_j+\sum_j c_j)}
   \\&<
   T \,\bigl[(1 - n^{\epsilon_2 - 1})^{-Mn^{\epsilon_0}\log n}-1\bigr]
   \\&<
   T \bigl[\exp\bigl(Mn^{\epsilon_0 + \epsilon_2 -1}\log n\bigr) - 1\bigr]
   \\&<
   TMn^{\epsilon_0 + \epsilon_2 -1}\log n.
\end{align*}
For the region
\(\lambda_in(1 - n^{-\epsilon})<S<\lambda_i n(1 + n^{-\epsilon})\)
and \(1 \le S_i \le n^{\epsilon_2}\) for each \(i\), we obtain
\[
   \sum _{i=1}^h\sum_{\,e\in A_i'',1\leq S_i\leq n^{\epsilon_2}} f(e)
   < TMn^{\,\epsilon_0 + \epsilon_2 - \epsilon}|h|\log n
   <
   TM^2n^{\,(\epsilon_0 + \epsilon_2 - \epsilon)}(\log n)^2.
\]
Similarly for the set \(B'\) with \(\bar S=\sum_j a_jh_j+\sum_j c_j l_j\), we have
\[
    \sum _{e\in B' 1\leq \bar{S} \leq n^{\epsilon_2}}f(e)< TMn^{\epsilon_0+\epsilon_2-\epsilon}.
\]
Define
\[
   \widehat{S}
   =\{
       e \in (\cup_i A_i'')\cup B' 
       : S_i > n^{\epsilon_2}\ \text{for all }i=1,\dots,h
         \;\;\text{and}\;\; \bar{S} > n^{\epsilon_2}
     \}.
\]
In this region, \eqref{eq:op_bd} and \eqref{eq:op_ac} give the upper bound
\[f(e) < n^{\tfrac{1}{3}n^{\epsilon_2}(\epsilon_2-1)}T.\]
This implies the upper bound for the region $\widehat{S}$
\begin{align*}   
   \sum_{\,e \in \widehat{S}} f(e)
   &<
   n^{\tfrac{1}{3}n^{\epsilon_2}(\epsilon_2-1)}\;
   T \;\cdot \binom{S'+Mn^{\epsilon_0}\log n }{Mn^{\epsilon_0}\log n}
   \\&<
   \exp\Bigl(
         C\,n^{\epsilon_0}\,(\log n)^2
         \;+\;
         (\epsilon_2-1)\,\tfrac13\,n^{\epsilon_2}\,\log n
       \Bigr)\;T
   \\&<
   \exp\bigl(-n^{\epsilon_2}\bigr)\,T
\end{align*}
when $\epsilon_2>\epsilon_0$.

Summing over all these regions shows that
\[
   \sum_{e} f(e)
   \;=\;
   \sum_{e\in A} f(e)\,\bigl(1 + O(n^{-\epsilon_3})\bigr),
   \qquad
   A = A_1 \,\cup\,\dots\,\cup\,A_h\,\cup\,B,
\]
for some \(\epsilon_3>0\).  
\end{proof}
\begin{rem}
    Also, using the same method we can prove a similar concentration phenomenon even if replacing some parts of elements to odd, for example, 
\[
g(a,b,c,d)
\;=\;
\mathbb{E}\Biggl[
\prod_{i=1}^h \frac{|\widehat{H}_i|^{2\,a_i+1}}{(2\,a_i+1)!}
\;\cdot\;
\frac{|\widehat{K}|^{2\,b+1}}{(2\,b+1)!}
\;\cdot\;
\prod_{i=1}^l \frac{|\widehat{L}_i|^{2\,c_i+1}}{(2\,c_i+1)!}
\;\cdot\;
\prod_{i=1}^m \frac{|\widehat{M}_i|^{2\,d_i+1}}{(2\,d_i+1)!}
\Biggr]
\]
also exhibit the concentration 
\[\sum_e g(e)= \sum_{e\in A}g(e)(1+O(n^{-\epsilon_3})).\]
\end{rem}
\noindent This proposition gives the following inequality about the partition function:
\[
   T\leq
   Z_n \le O(n^{M_0})\,T 
   \quad\text{for some constant } M_0>0.
\]
This proves the following proposition about the free energy.
\begin{prop}\label{NIM_free}
We have a free energy formula for the NIM model: \[
  F_n= \frac{1}{n}\log Z_n 
   =
   \frac{1}{n}\log T
   +O(n^{-\epsilon})=
   \frac{1}{n}\,\mathbb{E}\bigl[\exp(\widehat{H}_i)\bigr] +O(n^{-\epsilon}).
\]
\end{prop}
\subsection{Above threshold: The structure of Gibbs measure}
We apply the same notation for Hamiltonian above and we assume here that only one dominant term exists for Hamiltonian which will be used later to extend to heavy-tail case. More precisely, the Hamiltonian is 
\[
H_n(\sigma)=\;
\widehat{H}
+
\sum_{i=1}^l \widehat{L}_i
+
\sum_{i=1}^m \widehat{M}_i
\]
where $\widehat{H}=H\sigma_1...\sigma_pn^{-(p-2)/2}$ and $\widehat{L}_i,\widehat{M}_i$ are the same definition as above. In addition, it satisfies
\[
\mathbb{E}\bigl[\exp(\widehat{H})\bigr]
\;>\;
\exp(n^{1-2\epsilon})\,\mathbb{E}\bigl[\exp(\widehat{L}_i)\bigr],
\
\exp(n^{1-2\epsilon})\,\mathbb{E}\bigl[\exp(\widehat{M}_i)\bigr].
\]
\textbf{Remark}: Even without this assumption, we can obtain a similar linear combination result; however, we do not need such a case in this paper. Using the general version, we can observe the short moment when $\beta$ changes. In that short moment, both dominance occurs in the energy landscape and changes naturally to other dominance. This will be provided in our later work.
\\In the above assumptions, we prove this proposition about the limiting spin law and limiting Gibbs measure geometry. Before stating the propositions, we require more refined definitions to handle this limiting spin law in more detail. 
\begin{defn}\label{M+}
    We define $\R_{+}=\{x\in \R| x>0\}$, $\R_{-}=\{x\in \R| x<0\}$ and $M_+(H,p)=\{(\sigma_1,...,\sigma_p)|\sigma_i\in \{-1,+1\},H\sigma_1...\sigma_p>0\}$ and $f_{\pm}(x_1,...,x_p)=\{(\sigma_1,...,\sigma_n)| \sigma_i \in \R_{\text{sgn}(x_i)}\}$.
\end{defn}
Then we can state the following proposition.
\begin{prop}\label{Gibbs_upper}
For $x\in M_+(H,p)$,we have Gibbs measure concentration \[
    \langle\mathbbm{1}_{f_{\pm}(x)}\rangle=\frac{1}{2^{p-1}}(1-O(n^{-\epsilon})) \text{ and }\sum_{x\in M_+(H,p)}\langle \mathbbm{1}_{f_{\pm}(x)}\rangle =1-O(n^{-\epsilon}). 
\]
Moreover, in each region $x\in M_+(H,p)$, we have its conditional Gibbs measure concentration:
\[
    \langle |(\sigma_1,...,\sigma_ p )-\sqrt{t}x|^2|\sigma \in f_{\pm}(x)\rangle = O(n^{-\epsilon}).
\]
  We also have the limiting law:
    \[
        \frac{1}{\sqrt{n}}(\sigma_1,...,\sigma_p)\to [t(\xi_1,...,\xi_p)|H\xi_1...\xi_p>0]
    \]
    where $\xi_i$ is a random variable $\pm 1$ with probability 1/2 and $t=\frac{2\lambda_p( H)}{2p\lambda_p( H)+1}$.
    Morevoer, for two replicas, its restricted overlap satisfies
    \[
        \langle (\frac{1}{n}\sum_{i=p+1}^n\sigma_{i}^1\sigma_{i}^2)^2\rangle=O(n^{-\epsilon})
    \]
    where $\langle \rangle$ means expectation over the Gibbs measure.
\end{prop}
We adapt a similar representation for the Taylor expansion terms as \[
h(e)= \frac{\widehat{H}^{a}}{a!}
\;\cdot\;
\prod_{i=1}^l \frac{\widehat{L}_i^{b_i}}{b_i!}
\;\cdot\;
\prod_{i=1}^m \frac{\widehat{M}_i^{c_i}}{c_i!} \text{ and }
f(e)
\;=\;
\mathbb{E}\Biggl[
 \frac{\widehat{H}^{2a}}{(2a)!}
\;\cdot\;
\prod_{i=1}^l \frac{\widehat{L}_i^{2\,b_i}}{(2\,b_i)!}
\;\cdot\;
\prod_{i=1}^m \frac{\widehat{M}_i^{2\,c_i}}{(2\,c_i)!}
\Biggr]
\]
where $e=(a,b,c),b=(b_1,...,b_l),c=(c_1,...,c_m)$.
Moreover, the summation of $f$ is concentrated on tuple set 
\[
A=\{\lambda_p(H)n(1-n^{-\epsilon})<a<\lambda_p(H)n(1+n^{-\epsilon}),b=0,c_j\leq  n^{\epsilon} \text{ for all } j\}.
\]
Proposition \ref{conc_nim} implies 
\[
    \sum_{e}f(e)=\sum_{e\in A}f(e)(1+O(n^{-\epsilon})).
\]
Moreover, its replica also satisfies
\[
    \sum_{e,e'} \E[h(e)h(e')]= \sum_{e,e'\in 2A}\E[h(e)h(e')](1+O(n^{-\epsilon})).
\]
We define its restircted overlap as
\[
R_{1,2,p}=\frac{1}{n}\sum_{i=p+1}^n\sigma_{i}^1\sigma_{i}^2.
\]
Our objective value is to understand $\langle R^2_{1,2,p}\rangle$ and $\langle \sigma_i^2\rangle $ for $i=1,2,,,.p$. The basis for understanding this function is to understand $\E[\sigma_i^2\exp(H_n(\sigma))]$ and $\E[\sigma_i\sigma_j\sigma_i'\sigma_j'\exp(H_n(\sigma)+H_n(\sigma'))]$.
We first observe basis of its restricted overlap. The necessary condition to make this basis for overlap satisfies
\[
    \E[\sigma_i\sigma_j\sigma_i'\sigma_j'h(e)h(e')]\neq 0
\]
is $i=j$ and all the elements of $e$ and $e'$ are even or $i\neq j$ and $\widehat{M}_k=M_k\sigma_i\sigma_j$ holds for some k and $c_k,c_k'$ are odd and others are even. This easily yields 
\[
\sum_{e,e'} \E[R_{1,2,p}^2 h(e)h(e')]=\sum_{e,e'\in B} \E[R_{1,2,p}^2h(e)h(e')]+\sum_{k:\widehat{M}_k=M_k\sigma_i\sigma_j} \sum _{e,e'\in C_k}\frac{1}{n^2}\E[\sigma_i\sigma_j\sigma_i'\sigma_j'h(e)h(e')]
\]
where $B=\{(a,b,c)|a_i,b_i,c_i \text{ are even for all } i \}$ and\\ $C_k=\{(a,b,c)|a,b_i,c_j \text{ are even for all }i,j \text{ except } c_k \}$.
\\The left summation consists of the linear combination of \(
\E[\sigma_i^2\sigma_i'^2h(e)h(e')].
\)
The restricted overlap satisfies $|R_{1,2,p}|\leq 1$ and this gives
\begin{align*}
    \sum_{e\in B\text{ and }e\notin 2A \; or\; e'\notin 2A} \E[R_{1,2,p}^2 h(e)h(e')]&\leq \sum_{e\in B\text{ and }e\notin 2A \; or\; e'\notin 2A}\E[h(e)h(e')]\\&=O(n^{-\epsilon})\sum_{e,e'\in B}\E[h(e)h(e')].
\end{align*}
\\The remaining part in $B$ is just when $e,e'\in 2A$.Since $e,e'\in 2A$, we easily have
\[
    {\E[\sigma_i^2\sigma_i'^2h(e)h(e')]}\leq n^2 (n^{\epsilon}/n)^2{\E[h(e)h(e')]}
\]
and
\[
    \E[R_{1,2,p}^2h(e)h(e')]= \E[\frac{\sum_{i=p+1}^n (\sigma_i\sigma_i')^2}{n^2}h(e)h(e')]\leq n^{2\epsilon-1}\E[h(e)h(e')].
\]
In the remaining part, we have
\begin{align*}
    \sum_{k:\widehat{M}_k=M_k\sigma_i\sigma_j} \sum _{e,e'\in C_k}\frac{1}{n^2}\E[\sigma_i\sigma_j\sigma_i'\sigma_j'h(e)h(e')]&\leq \sum_{k:\widehat{M}_k=M_k\sigma_i\sigma_j} \sum _{e,e'\in C_k}\frac{1}{n^2}\E[|h(e)h(e')|]\\&\leq \frac{m}{n^2} \sum_{e,e'} \E[h(e)h(e')]=O(n^{-\epsilon})\sum_{e,e'} f(e)f(e').
    \end{align*}
By summing up all above and we have
\[
\sum_{e,e'} \E[R_{1,2,p}^2 h(e)h(e')]=O(n^{-\epsilon})\sum_{e,e'} f(e)f(e').
\]
This implies
\[
    \langle R_{1,2,p}^2\rangle=O(n^{-\epsilon}).
\]
Let $t=\frac{2\lambda_p(H)}{2\lambda_p(H)p+1}$,
we aim to show the limiting spin law. 
First, we prove
\(
\langle(\frac{\sigma_i^2}{n}-{t})^2\rangle =O(n^{-\epsilon})
\)
for $i=1,...,p$.
To do this, it is enough to calculate $\langle\sigma_i^2\rangle$ and $\langle\sigma_i^4\rangle$. We have 
\begin{align*}
\E[\sigma_i^2\exp(H_n(\sigma))]= \sum_e\E[\sigma_i^2h(e)]= \sum _{e\in 2A}\E[\sigma_i^2h(e)]+\sum_{e\notin 2A}\E[\sigma_i^2h(e)].
\end{align*}
Since we have \[\sum_{e\notin 2A}\E[\sigma_i^2h(e)]\leq nO(n^{-\epsilon})\sum_{e}\E[h(e)]\]
and 
\[
    \sum_{e\in 2A}\E[\sigma_i^2 h(e)]= n\sum_{e\in 2A}\E[h(e)]t(1+O(n^{-\epsilon})),
\]
we have
\[
    \E[\frac{\sigma_i^2}{n}\exp(H_n(\sigma))]=t(1+O(n^{-\epsilon}))\E[\exp(H_n(\sigma))].
\]
Similarly, we have \[\sum_{e\notin 2A}\E[\sigma_i^4h(e)]\leq n^2O(n^{-\epsilon})\sum_{e}\E[h(e)]\]
and 
\[
    \sum_{e\in 2A}\E[\sigma_i^4 h(e)]= n^2\sum_{e\in 2A}\E[h(e)]t^2(1+O(n^{-\epsilon})).
\]
This implies
\[
    \E[(\frac{\sigma_i^2}{n})^2\exp(H_n(\sigma))]=t^2(1+O(n^{-\epsilon}))\E[\exp(H_n(\sigma))].
\]
Above equalities imply 
\[
    \E[(\frac{\sigma_i^2}{n}-t)^2\exp(H_n(\sigma))]=O(n^{-\epsilon})\E[\exp(H_n(\sigma))].
\]
Thus, we have
\[
\langle (\frac{\sigma_i^2}{n}-t)^2\rangle =O(n^{-\epsilon}).
\]
The remaining part is now to show that this really converges to the random variable in the proposition. Due to Proposition \ref{conc_nim}, we have 
\[
\E[\exp(H_n(\sigma))\mathbbm{1}_{H\sigma_1...\sigma_p<0}]\leq \E[\exp (\sum_i \widehat{L}_i+\sum_i \widehat{M}_i)]=O(n^{-\epsilon})\E[\exp(H_n
(\sigma))].
\]
Then the Gibbs measure $G_n$ satisfies 
\[
G_n(\{(\sigma_1,...,\sigma_n)|H\sigma_1...\sigma_p>0\})=1-O(n^{-\epsilon}).
\]
Furthermore, we have Hamiltonian invariance under these transformations:
\(\sigma_i,\sigma_j\to -\sigma_i,-\sigma_j\) and for any $\pi \in S_p$, the transformation $\sigma_1,...,\sigma_p\to \sigma_{\pi(1)},...,\sigma_{\pi(p)}$.
We recall the definition $\R_{+}=\{x\in \R| x>0\}$, $\R_{-}=\{x\in \R| x<0\}$ and $M_+(H,p)=\{(\sigma_1,...,\sigma_p)|\sigma_i\in \{-1,+1\},H\sigma_1...\sigma_p>0\}$ and $f_{\pm}(x_1,...,x_p)=\{(\sigma_1,...,\sigma_n)| \sigma_i \in \R_{\text{sgn}(x_i)}\}$.
Then the invariance above implies that the value $G_n(f_{\pm}(x))$ is the same for every $x\in M_+$. More precisely, we have 
\[
    \E[\exp(H_n(\sigma))\mathbbm{1}_{f_{\pm}(x)}]
\]
is the same for every $x\in M_+$ and this implies
\[
\E[(\frac{\sigma_i^2}{n}-t)^2\exp(H_n(\sigma))\mathbbm{1}_{f_{\pm}(x)}]=O(n^{-\epsilon})\E[\exp(H_n(\sigma))\mathbbm{1}_{f_{\pm}(x)}].
\]
For $x\in M_+$, we also have 
\[
    \E[\exp(H_n(\sigma))\mathbbm{1}_{f_{\pm}(x)}]=\frac{1}{2^{p-1}}(1+O(n^{-\epsilon}))\E[\exp(H_n(\sigma))]. 
\]
and 
\[
    \langle (\frac{\sigma_i^2}{n}-t)^2\mathbbm{1}_{f_{\pm}(x)}\rangle =O(n^{-\epsilon}).   
\]
This implies 
\[
    \langle |(\sigma_1,...,\sigma_ p )-\sqrt{t}x|^2|\sigma \in f_{\pm}(x)\rangle = O(n^{-\epsilon})
\]
for every $x\in M_{+}$ .
If we fix $H$ and $n\to \infty$, the random vector converges in distribution as \[
        \frac{1}{\sqrt{n}}(\sigma_1,...,\sigma_p)\to [\sqrt{t}(\xi_1,...,\xi_p)|H\xi_1...\xi_p>0].
    \]
\section{Extension to heavy-tailed model}
In this section, we extend the result of the NIM model to our original heavy-tail model. We prove this energy landscape by dividing into 5 parts. We first prove that the largest $n^{\epsilon}$ interactions for each $p\leq M\log n $ becomes NIM model since there is no intersection between their spin indices with high probability. After that, we apply H\"older inequality and show other parts are small enough to negligible using heavy tail random matrix theory, coloring, concentration inequality. This extends the result of NIM model to our original heavy-tail model and this is why we call our method NIMR(Non-Intersecting Monomial Reduction).

Before stating more rigorous, we express the Hamiltonian as $\sum_{i,p} \widehat{H}_{i,p}$ where $\widehat{H}_{i,p}=\alpha(p)H_{i,p}\sigma_{i,p}$ and $\sigma_{i,p}=\sigma_{i,p,1}...\sigma_{i,p,p}n^{-(p-2)/2}$ and $H_{i,p}$ is ordered up to their absolute value for each $p$.(i.e., $|H_{1,p}|\geq |H_{2,p}|\geq...$)  
We divide its Hamiltonian into 5 parts as:
\[
H_n^A \;+\; H_n^B \;+\; H_n^C \;+\; H_n^D \;+\; H_n^E
\]
where
\[
\begin{aligned} 
&H_n^A: \text{NIM model,\quad collecting } i \le n^{\epsilon},p\le M\log n \\
&H_n^B: n^\epsilon \le i,\quad p=\text{$2$-spin},\\
&H_n^C: n^{\epsilon} \le i \le n^{\alpha+\epsilon},\quad 3 \le p \le M \log n,\\
&H_n^D: i \ge n^{\alpha + \epsilon},\quad 3 \le p \le M \log n,\\
&H_n^E: p \ge M \log n.
\end{aligned}
\]
We first show that the monomials in $H_n^A$ do not intersect each other with high probability and this becomes the NIM model with high probability. This ensures that we can directly invoke the properties of NIM model in the previous Section. Our strategy to prove the remaining part is to apply the H\"older inequality and divide them into 5 parts. $H_n^A$ part is the NIM model, we apply the SSK model with the heavy tail interaction method to the $H_n^B$ part, we apply the coloring argument to the $H_n^C$ part, we apply the concentration inequality on the sphere to the $H_n^D$ and the remaining $H_n^E$ is negligible using a simple calculation with mixed model properties.  
\subsection{NIM model regime and GSE}
We first define monomial graph. We apply a random graph argument to this monomial graph and show that the monomials in $H_n^A$ do not intersect each other with high probability.
Let us begin with showing that collecting the largest $n^{\epsilon}$ interaction ensures that every monomial is non-intersecting and forms the NIM model.
\begin{defn}[Monomial graph]
    The spin variables $(\sigma_1,...,\sigma_n)$ and subset of power set $\mathcal{I}$ of $\{1,...,n\}$. For any given constants $h_I$ and monomials $h_I\sigma_I=h_I\prod_{i\in I}\sigma_i$ for $I\in \mathcal{I}$, we can construct graph $G=(V,E)$ as $V$ is the set of monomials $\{h_I\sigma_I|I\in \mathcal{I}\}$ and $E=\{(h_I\sigma_I,h_J\sigma_J)|I\cap J\neq \emptyset\}$. We call this construction the monomial graph.  
\end{defn}
Using this definition and the random graph argument, we can prove that the $H_n^A$ part becomes the NIM model in the former section. Furthermore, this definition will be used later for the coloring argument in the next subsection.
\begin{lem}
For $\epsilon<1/2$,
    if we collect $n^{\epsilon}$ interactions with the largest absolute values for each $p$, then all the corresponding monomials are non-intersecting with high probability. More precisely, the probability that an intersection occurs is $O(n^{2\epsilon'-1})$ for every $\epsilon'<\epsilon$.
\end{lem}
\begin{proof}
    We focus on the regime $p < M \log n$. Let $\widetilde{H}_{n,p}$ be the $p$-spin 
Hamiltonian defined by
\[
  \widetilde{H}_{n,p} 
  \;=\; \sum_{i=1}^{\binom{n}{p}} 
        H_{i,p}\,\sigma_{i,p,1}\,\sigma_{i,p,2}\,\cdots\,\sigma_{i,p, p}
        \;n^{-\frac{p-2}{2}},
\]
where each $H_{i,p}$ is a coefficient, and we define
\[
   \Delta_{i,p} \;=\; 
     \{\sigma_{i,p,1},\,\sigma_{i,p,2},\,\dots,\,\sigma_{i,p,p}\}.
\]
Now consider the following monomial graph construction.  
Define
\[
  G \;=\; (V, E),
\]
where
\[
  V \;=\; 
   \bigl\{\,\Delta_{i,p} \;\big|\; p < M \log n,\; i < n^{\epsilon}\bigr\}
\]
and
\[
  E 
   \;=\; \bigl\{\{\Delta_{i,p}, \Delta_{j,s}\} 
         \;\big|\; \Delta_{i,p} \cap \Delta_{j,s} \neq \emptyset\bigr\}.
\]
For each fixed $p$, the set $\Delta_{i,p}$ is chosen uniformly at random from 
\(\binom{[n]}{p}\).  Then the probability that two such sets intersect is
\[
  \P(\Delta_{i,p} \cap \Delta_{j,s} \neq \varnothing) 
   \;=\; O\!\bigl(\tfrac{(M \log n)^2}{n}\bigr).
\]
Hence the expected number of edges in \(G\) is
\[
  \mathbb{E}[|E|] 
  \;=\; O\!\Bigl(\tfrac{(M \log n)^2}{n}\Bigr) 
        \times (M \log n \, n^{\epsilon})^2 
  \;=\; O\bigl((M \log n)^4 \, n^{2\epsilon-1}\bigr).
\]
If we choose \(\epsilon < \tfrac12\), then \(2\epsilon - 1 < 0\) and $\P(|E|\geq 1)\leq \E[|E|]=O(n^{2\epsilon'-1})$ for any $\epsilon'<\epsilon$.
It follows that with high probability \(E\) is empty and all the sets in \(V\) are pairwise disjoint.
\end{proof}
Furthermore, we need the property that our interaction is not too close to the threshold for below the threshold case and the most dominant interaction and the second dominant interaction are not too close to apply Proposition \ref{Gibbs_upper}.
\begin{lem}
    We have the following with high probability:
    \begin{itemize}
        \item For $A\in \mathcal{F}_1 $, every $p$-interaction is under the $H_p^{*}-n^{-\epsilon}$ for all $p\leq M\log n$.
        \item For $A\in \mathcal{F}(H,p)$, we have \[|H_{1,q}|<f_{q}^{-1}f_p(H)-n^{-\epsilon}, \text{ and } |H_{2,p}|<|H_{1,p}|-n^{-\epsilon}\]
        for all $q\leq M\log n$.
    \end{itemize}
\end{lem}
Above two lemmas imply that the model becomes NIM model and it satisfies the condition to propositions in the former section.
Before we extend this to the whole part, we first prove that $H_n^E$ is small enough and we prove Theorem \ref{GSE} which shows ground state energy has $O(n)$ size fluctuation and this ensures that the normalization factor is proper.
Before we prove these, we need more procedures to understand the heavy-tail distributions. This lemma ensures that upper bound of orderd statics of heavy tail distributions with high probability.
\begin{lem}
Let \(X_1,\dots,X_n\) be i.i.d.\ heavy-tailed random variables with exponent \(\alpha\). 
Denote by \(|Y_1| \ge |Y_2| \ge \dots \ge |Y_n|\) their order statistics. \
If $m>n^{\epsilon'}$ and $a>n^{\epsilon}{(n/m)^{1/\alpha}}$ hold for some $\epsilon,\epsilon'>0$ then we have 
\[
    \P(|Y_m|>a)\leq \exp(-cm)
\]
for some $c>0$.
\end{lem}
\begin{proof}
Define \(P_a = \mathbb{P}(|X| > a)\). 

\medskip

Then
\[
\mathbb{P}(|Y_m| < a) 
= \sum_{k=0}^{m} \binom{n}{k} P_a^k \bigl(1-P_a\bigr)^{n-k}.
\]

\medskip

Consider the binomial random variable \(B_n \sim \mathrm{Bin}(n,P_a)\). We have 
\[
\mathbb{E}[B_n] = nP_a, 
\quad 
\mathrm{Var}(B_n) = nP_a\bigl(1-P_a\bigr), \quad \mathbb{P}(|Y_m| > a) 
= \mathbb{P}(B_n > m).
\]
Use Chernoff bound and we have 
\[
\P(B_n>m)\leq O(\exp(-c_{\delta}nP_a))
\]
where $\delta=\frac{m-nP_a}{nP_a}$.
This shows \[ \P(|Y_m|>a)\leq \exp(-cm) \] since $a> n^{\epsilon} (n/m)^{1/\alpha}$ and this implies $nP_a\ll m$. 
\end{proof}
As a corollary, we define the event \[A:=\{\text{There exists } i\geq n^{\epsilon},p \text{ such that } |H_{i,p}|\geq n^{\epsilon_0}i^{-1/\alpha}\}.\] Then, we have 
    \[
    \P(A)\leq M\log n \cdot n^{M\log n}\exp(-cn^{\epsilon})\leq \exp(-c'n^{\epsilon}).
    \]
\begin{lem}\label{heavy_upper}
    We have upper bounds with probability:\[
    |H_{i,p}|<n^{\epsilon_0}i^{-1/\alpha}
    \]
    for $i>n^{\epsilon}$.
    More precisely, the probability that this does not hold is less than $\exp(-c'n^{\epsilon})$.
\end{lem}
Using this lemma and a simple probabilistic argument, we can prove the upper bound of $H_n^E$ with high probability.
\begin{prop}\label{upper_E}
For any $M'>0$, if we choose $M$ is large enough, then we have $|H_n^{E}|<n^{-M'}$ with high probability. 
\end{prop}
\begin{proof}
For mixed $p$ spin spherical model, we have $\sum_{p\geq 2} \alpha(p)(1+\epsilon)^p<\infty$ for some $\epsilon>0$.
    For each $p$,we define $|H_{1,p}|$ as the largest absolute interaction and we have \[\P(|H_{1,p}|<b_{n,p}m)=(1- L(b_{n,p}m)b_{n,p}^{-\alpha}m^{-\alpha})^{\binom{n}{p}}=\exp(O(-m^{-\alpha}))=1-O(m^{-\alpha}).\]
Hence we have union bound for 
\[\P(|H_{1,p}|<m_p \text{ for every }p\geq M\log n )\geq 1-O(\sum_{p\geq M\log n} m_p^{-\alpha}). \]
If we choose $m_p=(1+\epsilon)^{p/100}$, then $\sum_{p\geq M\log n} m_p^{-\alpha}=O((1+\epsilon)^{-M\log n /100}).$
For large enough $M$,
\[|\sum_{p\geq M \log n } \alpha(p)H_{i,p}\sigma_{i,p}| \leq 
|\sum_{p\geq M\log n } \alpha(p) n (\sum_i H^2_{i,p})^{1/2}|
\]
We have 
\[
\sum_i H_{i,p}^2<\sum_{i\leq n^{\epsilon_0}}m_p^2+\sum_{k>n^{\epsilon_0}}n^{\epsilon_0(\frac{1}{\alpha}-\frac{1}{2})}k^{-\frac{2}{\alpha}}=n^{\epsilon_0}(1+\epsilon)^{p/50}+O(n^{\frac{1}{2}\epsilon_0(\frac{1}{2}-\frac{1}{\alpha}})).
\]
Thus, we have
\[
\sum_{p\geq M\log n }n^{1+\frac{\epsilon_0}{2}}(1+\epsilon)^{p/100}\alpha(p)\leq \sum_{p\geq M\log n } n^{1+\frac{1}{2}\epsilon_0}C(1+\epsilon)^{-\frac{99}{100}p}=O((1+\epsilon)^{-\frac{M}{2}\log n }n^{1+\frac{\epsilon_0}{2}})=O(n^{-M'})
\]
if $M$ is large enough.
\end{proof}
Now we can prove Theorem \ref{GSE}.
\begin{proof}[Proof of Theorem \ref{GSE}]
    We first prove that $\frac{1}{n}\max H_n^A= \max_{p}|\alpha(p)H_{1,p}|p^{-p/2}$. Using AM-GM inequality, we have
    \[|H_{n}^A(\sigma)|=|\sum_{2 \leq p \leq M\log n }\sum_{i=1}^{n^{\epsilon}} \alpha(p)H_{i,p}\sigma_{i,p}|\leq |\sum_{2\leq p \leq M\log n,\; i\leq n^{\epsilon}} \alpha(p) H_{i,p}(\frac{
\sum_{k=1}^p \sigma^2_{i,p,k}}{p})^{p/2}|n^{-(p-2)/2}.\]
Using the fact that \[\sum_{i\leq n^{\epsilon},2\leq p \leq M\log n } \sigma_{i,p,k}^2\leq n \]
and convexity, we have
\[
    |\sum_{2\leq p \leq M\log n,\; i\leq n^{\epsilon}} \alpha(p) H_{i,p}(\frac{
\sum_{k=1}^p \sigma^2_{i,p,k}}{p})^{p/2}n^{-(p-2)/2}|\leq \max _{p}|H_{1,p}\alpha(p)p^{-p/2}|n
\]
and the equality holds only when $|\sigma_{1,p,k}|= \sqrt{n/p}$ for all $k=1,...,p$.
Hence, we have $\frac{1}{n}\max H_n(A)=\max_p |\alpha(p) H_{1,p}|p^{-p/2}$. 
The remaining part is to show that $\sum_{i>n^{\epsilon},p<M\log n }\widehat{H}_{n,p}$ is small enough. Due to Lemma \ref{heavy_upper}, we have $H_{i,p}\leq n^{\epsilon_0}i^{-1/\alpha}$ and this implies $\sum_{i\geq n^{\epsilon}} H_{i,p}^2\leq n^{2\epsilon_0}n^{\epsilon(1-\frac{2}{\alpha})}$ for each $p$. Since $1-2/\alpha<0$, if we choose $\epsilon_0>0$ is small enough, we can ensure that $(\sum_{i\geq n^{\epsilon}} H_{i,p}^2)^{1/2}< n^{-\epsilon'}$ for some $\epsilon'>0$ for all $i\leq M\log n$.
We apply Cauchy-Schwartz inequality and we  have \[
|\sum_{i>n^{\epsilon}}\widehat{H}_{i,p}|\leq \alpha(p) (\sum_{i\geq n^{\epsilon}}H_{i,p}^2)^{1/2}(\sum_{i\geq n^{\epsilon}}\sigma_{i,p,1}^2...\sigma_{i,p,p}^2)^{1/2}n^{-(p-2)/2}\leq \alpha(p)n (\sum_{i\geq n^{\epsilon}}H_{i,p}^2)^{1/2}
\]
since 
\[
    \sum_{i\geq n^{\epsilon}}\sigma_{i,p,1}^2...\sigma_{i,p,p}^2\leq (\sum_{i=1}^n \sigma_i^2)^p=n^p.
\]
Since we have $(\sum_{i\geq n^{\epsilon}} H_{i,p}^2)^{1/2}< n^{-\epsilon'}$, we can prove that 
\[
|\sum_{i>n^{\epsilon}}\widehat{H}_{i,p}|\leq \alpha(p)n (\sum_{i\geq n^{\epsilon}}H_{i,p}^2)^{1/2}< \sum_{p\leq M\log n } \alpha(p)n^{1-\epsilon'}=O(n^{1-\epsilon'}).
\]
Therefore, we have 
\[
    \frac{1}{n}\max |H_n^A|=\max_{p}|\alpha(p)H_{1,p}p^{-p/2}| 
\]
and
\[
    |H_n^B+H_n^C+H_n^D|=O(n^{1-\epsilon'})
\]
and 
\[
    |H_n^E|=O(n^{-M'})
\]
due to Proposition \ref{upper_E}.
Combining all above, we have
\[
    \frac{1}{n}GSE=\max_{p}|\alpha(p)H_{1,p}p^{-p/2}| +O(n^{-\epsilon'}).
\]
\end{proof}
\subsection{The heavy tail matrix regime and the coloring argument}
Now we prove that the part $H_n^B$ is small enough using \cite{kim2024fluctuations}. The main idea of the paper is understanding eigenvalue behavior of the heavy tail random matrix and we apply such idea.
\begin{lem}\label{upper_B}
    Hamiltonian $H_n(\sigma)= n^{\epsilon_0}\sum_{i\geq n^{\epsilon}} H_{i,2}\sigma_{i,2}$ is given, then there is small enough $\epsilon_0$ 
    such that there exists $\epsilon'>0$ which satsifes
    \[\E[\exp(H_n(\sigma))]=1+O(n^{-\epsilon'}).\]
\end{lem}
\begin{proof}
Since the Hamiltonain consists of degree 2 monomials in $\sigma_1,...,\sigma_n$, we can consider $n\times n$ symmetric matrix $M_{ij}=\frac{1}{2}f_{ij}$ where $f_{ij}=H_{ij}$ if $H_n(\sigma)$ has $H_{ij}\sigma_{i}\sigma_j$ otherwise $f_{ij}=0$. Then we can apply the same method in the Section 4 (High tempreature regime) of \cite{kim2024fluctuations} and we have
\[
    \log \E[\exp(H_n(\sigma))]= -\frac{1}{2}\sum_{i}\log (1-2\lambda_i )+O(n^{-\epsilon})
\]
for eigenvalues $\lambda_1\geq \lambda_2\geq...\geq \lambda_n $ of $M$.
Due to Lemma \ref{heavy_upper}, we have
$\lambda_i<n^{-\epsilon_1} $ for all $i$ and for some $\epsilon_1>0$.
Then its Taylor expansion implies
\[
    -\frac{1}{2}\log(1-2\lambda_i)=\lambda_i+\lambda_i^2(1+O(n^{-\epsilon
_1}))
\]
and
\[
    \sum_i-\frac{1}{2}\log(1-2\lambda_i)= \sum_i\lambda_i+\sum_{i}\lambda_i^2(1+O(n^{-\epsilon_1}))=\sum_i \lambda_i^2 (1+O(n^{-\epsilon_1}))
\]
since $Tr M=0$.
The last term satisfies \[\sum_i \lambda_i^2=\sum_i M_{ij}^2=\frac{1}{2}\sum_{i\leq j } H^2_{ij}=O(n^{-\epsilon'})\] with high probability.
This shows 
\[
    \log \E[\exp (H_n(\sigma))]=O(n^{-\epsilon'})
\]
and 
\[
    \E[\exp(H_n(\sigma))]=1+O(n^{-\epsilon'}).
\]
\end{proof}
Furthermore, we prove that the $H_n^C$ part is small enough using the coloring argument. To prove this, we require these extremal combinatorics arguments. 

\begin{prop}\label{coloring}
    Consider set $[n]=\{1,2,...,n\}$ and consider graph $G=(V,E)$ where
    $V=\{A|A\subset [n] \text{ and }|A|=p\}$ and $E=\{(A,B)|A\cap B\neq\emptyset\}$. For large enough $n$ and $p\leq M\log n$ and randomly chosen sets $V_1,...,V_{n}\subset V$ with $|V_i|=n^{1-s}$ and for their induced subgraph $G_i$, we have
    \[
    \P( \text{All }G_i \text{ are } \frac{2}{s}+1 \text{ colorable } )>1- n^{-1}.
    \]
\end{prop}
To prove this proposition, we first show the following lemmas about upper bound of cycle length which will give the upper bound of the coloring. Using this and counting expected number of cycle, we can prove the following proposition.  
\begin{lem}
\label{lem:degree-cycle}
Let $G$ be a graph in which every vertex has degree at least $k$. Then $G$ contains a cycle of length at least $k+1$.
\end{lem}

\begin{proof}
Consider a longest path in $G$, say $v_1, v_2, \ldots, v_t$.  Because $v_1$ has degree at least $k$, it has a neighbor in $\{v_2,\dots,v_t\}$ that is at distance at least $k$ along the path.  Concretely, there must be some $v_s$ with $s \ge k+1$ such that $(v_1,v_s)$ is an edge.  Hence $v_1, v_2, \ldots, v_s$ form a cycle of length at least $k+1$.
\end{proof}

\begin{lem}
\label{lem:bounded-degree-color}
Let $G$ be a graph such that every subgraph of $G$ contains a vertex of degree at most $k$. Then $G$ is $(k+1)$-colorable.
\end{lem}

\begin{proof}
We proceed by induction on $|V(G)|$.  If $|V(G)| \le k+1$, then $G$ can trivially be colored with at most $k+1$ colors, so $\chi(G) \le k+1$.

Now assume the claim holds for all graphs of smaller size.  Suppose $|V(G)| = n$ and let $v$ be a vertex of degree at most $k$ (which exists by hypothesis).  Consider the subgraph $G - v$.  By induction, $G - v$ is $(k+1)$-colorable.  Because $\deg(v) \le k$, vertex $v$ can be assigned one of the $(k+1)$ colors unused in its neighborhood.  Thus $G$ is also $(k+1)$-colorable.
\end{proof}

\begin{lem}
\label{lem:no-long-cycle-color}
If a graph $G$ has no cycle of length $\ge k+1$, then it is $(k+1)$-colorable.
\end{lem}

\begin{proof}
Since $G$ has no cycle of length at least $k+1$, the same holds for every subgraph of $G$.  Therefore, by the contrapositive of Lemma~\ref{lem:degree-cycle} and Lemma~\ref{lem:bounded-degree-color}, we conclude $G$ is $(k+1)$-colorable.
\end{proof} 
Now, we are prepared to prove the proposition using the arguments above. 
\begin{proof}[Proof of Proposition \ref{coloring}]
For given two vertices, the probability that it has intersection is $O(p^2/n)$. For each given $k$ vertices $v_1,...,v_k$, the probability that $v_1v_2...v_k$ forms a cycle is $O(p^{2k}/n^k)$. Therefore, the number of length $k-$cycle in $G_1\cup G_2\cup..\cup G_n$ is
\begin{align*}
     \E[\#\text{of length k cycles in }G_1\cup G_2\cup..\cup G_n ]&\leq n\cdot \frac{n^{1-s}(n^{1-s}-1)...(n^{1-s}-k+1)}{k} O(p^{2k}/n^k)\\&=O(p^{2k}n^{1-ks}).
\end{align*}
Since, $p\leq M\log n $, if we choose $k=\frac{2}{s}$, then $O(p^{2k}n^{1-ks})<n^{-1}$ and this also implies that 
\[
\E[\#\text{of length}\geq k \text{ cycles in }G_1\cup G_2\cup..\cup G_n ]\leq \sum_{k\geq 2/s} O(p^{2k}n^{1-ks})<n^{-1}.
\]
Hence, $G_1\cup G_2\cup..\cup G_n$ has no cycle length $>\frac{2}{s}$ with probability $>1-n^{-1}$. This implies it is $\frac{2}{s}+1$ colorable with probability $>1-n^{-1}$.
\end{proof}
This proposition gives upper bound for $H_n^C$.
\begin{prop}\label{upper_C}
There exists small enough $\epsilon_0>0$ such that for every $s<n^{\epsilon_0}$, we have 
    \[\E[\exp (sH_n^C)]<1+n^{-\epsilon'}\]
    for some $\epsilon'>0$.
\end{prop}
\begin{proof}
    Let \(H_{n,p}^c\) be the $p$-spin part of the Hamiltonian \(H_n^c\).
\\Applying H\"older inequality and we have 
\[
\mathbb{E}\!\Bigl[\exp\!\Bigl(s \sum_{p=2}^{M\log n } H^C_{n,p}\Bigr)\Bigr]
\;\le\;
\max_{p}\;
\mathbb{E}\!\Bigl[\exp\!\bigl(s\,M \log n\,H^C_{n,p}\bigr)\Bigr].
\]
Hence, it is enough to show the upper bound for each $p$. We first divide $H_{n,p}^C$ into two parts using H\"older inequality. 
\[
    \E[\exp(sH_{n,p}^C)]\leq \E[\exp(s(1+n^{-\epsilon})H_{n,p}^{C,1}]^{\frac{n^{\epsilon}}{1+n^{\epsilon}}}\E[\exp(s(1+n^{\epsilon}) H_{n,p}^{C,2}]^{\frac{1}{1+n^{\epsilon}}}
\]
where
\[
    H_{n,p}^{C,1}=\sum_{n^{\epsilon}\leq i<a }\widehat{H}_{i,p} \text{ and } H_{n,p}^{C,2}=\sum_{a< i\leq a(b+1)-1 }\widehat{H}_{i,p}.
\]
For the first part, we have 
\[
    \E[\exp (s H_{n,p}^{C,1})]\leq \prod_{i=n^{\epsilon}}^a \E[\exp(s\widehat{H}_{i,p})]\leq 1+n^{2-p}s^2 \sum_{n^{\epsilon}\leq i <a}{H^2_{i,p}}=1+O(n^{2-p}s^2 n^{\epsilon_0}n^{\epsilon(1-\frac{2}{\alpha})}).
\]
We apply H\"older inequality and we have
\[
\mathbb{E}\!\Bigl[\exp\!\bigl(sH_{n,p}^{C,2}\bigr)\Bigr]
\;\le\;
\prod_{k=1}^{b} \mathbb{E}\!\Biggl[\exp\!\Bigl(\frac{\sum_{l=1}^b H_{al,p}}{H_{ak,p}}s \sum_{i = a\,k}^{\,a(k+1)-1} H_{i,p}\,\sigma_{i,p}\Bigr)\Biggr]^{\frac{H_{ak,p}}{\sum_{l=1}^b H_{al,p}}}.
\]
We set
\[
a \;=\; n^{1 - \epsilon},
\quad
b \;=\; n^{\alpha-1 + 2\epsilon}.
\]
and $\epsilon$ is small enough so that $\alpha-1 + 2\epsilon<1$.
Now we can apply Proposition \ref{coloring} and prove that the monomial graphs of $A_k=\{\sigma_{i,p}| ak\leq i\leq a(k+1)-1\}$ are $C(\alpha)$ colorable. More precisely, for any given Hamiltonian $H_n(\sigma)=\sum_i \widehat{H}_i$, if its monomial graph is $C-$colorable and we define $\mathcal{C}$ to be the set of color and we apply H\"older inequality
\[
    \E[\exp(\sum_i \widehat{H}_i ) ]\leq \prod_{c\in \mathcal{C}}\E[\exp(C\sum_{\widehat{H}_i \text{ is colored as c } }\widehat{H}_i)]^{1/C}\leq \prod_i \E[\exp(C\widehat{H}_i)].
\]
After applying this inequality, we can apply the NIM argument under the threshold as in Proposition \ref{under_free} if $s\frac{\sum_{l=1}^b H_{al,p}}{H_{ak,p}}H_{i,p}\leq s \sum_{l=1}^b H_{al,p}$ is under the threshold. This shows

\begin{align*}
    \mathbb{E}\!\Bigl[\exp\!\bigl(sH_{n,p}^{C,2}\bigr)\Bigr]&\leq \prod_{k=1}^{b}
 \Bigl[
   1
   +
   s^{2}n^{\,2 - p}C(\alpha)^2a \frac{(\sum_{l=1}^b H_{al,p})^2}{H_{ak,p}^2}\sum_{i=ak}^{a(k+1)-1} H^2_{i,p}(1+O(n^{-\epsilon}))
 \Bigr]\\ &\leq \prod_{k=1}^{b}[1+s^2n^{2-p}C(\alpha)^2 aH_{ak,p}\sum_{l=1}^b H_{al,p}(1+o(n^{-\epsilon})]
\\&\leq\exp(s^2n^{2-p}C(\alpha)^2a(\sum_{l=1}^b H_{al,p})^2(1+O(n^{-\epsilon}))).
\end{align*}

Since we have 
\(
    H_{k,p}< k^{-1/\alpha}n^{\epsilon_0}
\) with high probability,
we have
\[
a
 \Bigl(\sum_{l=1}^{b}H_{al,p}\Big)^{2}
 \leq
 a^{1 - \frac{2}{\alpha}}
 n^{\epsilon_0}
 b^{2 - \frac{2}{\alpha}}
 =
 n^{\epsilon_0}
 n^{2\alpha - 3 +\epsilon\bigl(3 - \frac{2}{\alpha}\bigr)}
 <
 n^{1 - \epsilon_{1}}
\]
and 
\[
s\frac{\sum_{l=1}^b H_{al,p}}{H_{ak,p}}H_{i,p}\leq s \sum_{l=1}^b H_{al,p}\leq s n^{\epsilon_0}a^{-1/\alpha}b^{1-1/\alpha}=sn^{\epsilon_0+\alpha-2-\frac{\epsilon}{\alpha}}
\]
for suitably chosen \(\epsilon,\epsilon_0,\alpha\). This concludes that
\[
\mathbb{E}\!\Bigl[\exp\!\bigl(sH_{n}^C\bigr)\Bigr]
 \leq1 +
 O\!\bigl(n^{-\epsilon_{1}/3}\bigr),
\]
provided 
\(\epsilon_{2} < \frac{\epsilon_{1}}{3}\), \(p \ge 3\), and when \(s<n^{\,\epsilon_{2}}\). 
\end{proof}
\subsection{Concentration of the measure on sphere}
The last part is to show that $H_n^D$ is small enough using concentration inequality and conditional expectation procedure.  
\begin{prop}\label{upper_D}
    There exists $\epsilon'''>0$ such that \[
    \E[\exp(sH_n^D)]<1+O(n^{-\epsilon'})
    \]
    holds for all $s<n^{\epsilon'''}$with high probability.
\end{prop}
\begin{proof}

Now we use sightly different notion. Our usual $\E$ which is the expectation over sphere will be expressed as $\E_{\sigma}$ and expectation over $H_{i,p}$'s will be expressed as $\E_{H}$. Moreover, we define $\mathcal{{F}}_{ABC}$ which is event that coefficients in $H_n^A,H_n^B,H_n^C$ are already observed. 
Our objective value is
\[
  \mathbb{E}_H\,\mathbb{E}_\sigma\Bigl[\exp\bigl(sH_n^D\bigr)
  \;\Big|\;\mathcal{F}_{ABC}\Bigr].
\]
Since we can exchange the order due to Fatou's Lemma,
we have 
\[
  \mathbb{E}_H\mathbb{E}_\sigma\Bigl[\exp\bigl(sH_n^D\bigr)
  \;\Big|\;\mathcal{F}_{ABC}\Bigr]=\mathbb{E}_\sigma\mathbb{E}_H\Bigl[\exp\bigl(sH_n^D\bigr)
  \;\Big|\;\mathcal{F}_{ABC}\Bigr]= \mathbb{E}_\sigma\Bigl[\prod_{i,p}
  \E_H\bigl[\exp(s\widehat{H}_{i,p})\big|\mathcal{F}_{ABC}\bigr]\Bigr].
\]
For the convenience, we omit $\mathcal{F}_{ABC}$ and just consider $\E_H$ as this conditional expectation.
If we see the distribution $H_{i,p}$ for all $i\geq n^{\alpha+\epsilon}$ under the condition, the collective distribution is $i.i.d $ distributions with each variable follows conditional distribution $H_p'=[H||H|\leq H_{n^{\alpha+\epsilon},p}] $ where $H$ is the heavy tailed random variable which every interaction follows. 
This implies
\[
\E_H\bigl[\exp(sH_n^D)\big|\mathcal{F}_{ABC}\bigr]=\prod_{i,p}\E_H[\exp(sH_p'\sigma_{i,p})].
\]
For each term, 
\[
\E_H[\exp(sH_p'\sigma_{i,p})]= \sum_{k=0}^{\infty} \frac{1}{k!}\E[s^kH_p'^k]\sigma_{i,p}^k=1+s\E[H_p']\sigma_{i,p}+\frac{1}{2}s^2\E[H_p'^2]\sigma^2_{i,p}(1+O(n^{-\epsilon'})).
\]
The last term is induced from the fact that $sH_p'\sigma_{i,p}=O(n^{-\epsilon'})$ coming from the fact that $\sigma_{i,p}= O(n)$ and Lemma \ref{heavy_upper} implies $H_{p}'\leq n^{-1-\epsilon'}$ and choosing $s< n^{\epsilon''}$ for small enough. Then, we define $H_p^D$ as $p-$spin part of $H_n^D$ and we have 
\begin{align*}
   \E_{\sigma}\E_H[\exp (sH_p^D)]&\leq \E_\sigma[ \exp(s\sum_{i\geq n^{\alpha+\epsilon}}\E[H_p']\sigma_{i,p}+\frac{1}{2}s^2\E[H_p'^2]\sum_{i\geq n^{\alpha+\epsilon}}\sigma_{i,p^2}(1+O(n^{-\epsilon'}))) ]
   \\&\leq \E_\sigma[\exp(s\sum_{i\geq n^{\alpha+\epsilon}}\E[H_p']\sigma_{i,p}+\frac{1}{2}s^2\E[H'^2_{p}]n)]
\end{align*}
Here we need upper bound for $\E[H_p']$ with high probability.
If $\alpha<1$, we have \[
\E[H_p']\leq b_{n,p}^{-1}\cdot O(n^{p(1-\alpha)})\leq n^{\epsilon_0+p(1-\alpha-1/\alpha)}< n^{-p}.
\] 
If $\alpha=1$, we have
\[
    \E[H_p']\leq b_{n,p}^{-1}O(n^{\epsilon'})=O(n^{-p+\epsilon'}).
\]
If $\alpha>1$, we have
\[
    \E[H_p']\leq b_{n,p}^{-1}O(1)=O(n^{-p/\alpha +\epsilon'}).
\]
For any cases, we have 
\[
    \E[H_p']n^{-(p-2)/2}< n^{1-p+\epsilon'}. 
\]
Furthermore, we have
\[
    \E_{\sigma}[\exp(s\sum_{i\geq n^{\alpha+\epsilon}} \E[H_p']\sigma_{i,p})]\leq \E_{\sigma}[\exp(s\E[H_p'](\sigma_1+...+\sigma_n)^p)].
\]
To obtain the upper bound, we use concentration of the measure on the sphere.
Since the function \(\frac{\sum \sigma_i}{\sqrt{n}}\) is 1-Lipshitz, concentration of measure on sphere(Theorem 2.3 in \cite{ledoux2001concentration}) implies
\[
  \P\!\Bigl(\bigl|\sigma_{1} + \sigma_{2} + \dots + \sigma_n\bigr| > m\Bigr)
  \;\le\;
  2\exp\bigl(-\frac{(n-1)m^2}{2n^2}\bigr).
\]
Since it is on sphere, all the event satisfies \(\{\,|\sigma_{1} + \dots + \sigma_n|\le n\}\).  Suppose
\(A < n^{1 - p - \epsilon}\), then
\begin{align*}
  \mathbb{E}\!\bigl[\exp\bigl(A(\sigma_{1}+\dots+\sigma_n)^p\bigr)\bigr]
  &=
  \mathbb{E}\!\Bigl[\exp\!\bigl(A(\sigma_{1}+\dots+\sigma_n)^p\bigr)
    \,\mathbf{1}_{\{|\sigma_{1}+\dots+\sigma_n|<n^{\tfrac12+\epsilon_0}\}}\Bigr]
  \\&+
  \mathbb{E}\!\Bigl[\exp\!\bigl(A(\sigma_{1}+\dots+\sigma_n)^p\bigr)
    \,\mathbf{1}_{\{|\sigma_{1}+\dots+\sigma_n|\ge n^{\tfrac12+\epsilon_0}\}}\Bigr]
    \\&\leq \exp(n^{-\frac{p}{2}-\epsilon+\frac{p\epsilon_0}{2}})+n\exp(-n^{2\epsilon_0})=1+O(n^{-\epsilon'}).
\end{align*}
Therefore, 
\begin{align*}
   \E_{H}\E_{\sigma}[\exp (sH_p^D)]&\leq \E_\sigma[\exp(s\sum_{i\geq n^{\alpha+\epsilon}}\E[H_p']\sigma_{i,p}+\frac{1}{2}s^2\E[H'^2_{p}]n)]
   \\&\leq (1+O(n^{-\epsilon'}))\exp(s^2n^{-1-2\epsilon'})=1+O(n^{-\epsilon''}).
\end{align*}
We also know that $\exp(x)\geq 1+x$ implies $\E_{\sigma}[\exp (sH_p^D)]\geq 1+\E[sH_p^D]=1$. Since the probability of the event $\E_{\sigma}[\exp (sH_p^D)]\geq 1+O(n^{-\epsilon''/2})$ is $O(n^{-\epsilon''/2})$. 
\\This easily implies \[\P(\E_{\sigma}[\exp (sH_p^D)]< 1+O(n^{-\epsilon''/2})\text{ holds for all }p\leq M\log n )=1-O(n^{-\epsilon''/3}).\]
We apply H\"older inequality and we have
\[
\E_{\sigma}[\exp (sH_n^D)]\leq \E[\exp(sM\log n \exp H_p^D)]^{1/M\log n}\leq \max_p\E[\exp(sM\log n \exp H_p^D)].
\]
Hence, if we choose $s< n^{\epsilon'''}$, we can easily prove that
 $\E_{\sigma}[\exp (sH_n^D)]< 1+O(n^{-\epsilon''/2})$ with high probabiilty.
\end{proof}
\subsection{Proof of the main theorems}
Now we are prepared to prove extension to our original heavy tail model. 
\begin{proof}[Proof of Theorem \ref{Free}]
    We apply H\"older inequality and we have 
    \begin{align*}
        \E[\exp(H_n(\sigma))]&\leq \E[\exp (\frac{1}{S_1}H_n^A)]^{S_1}\E[\exp (\frac{1}{S_2}H_n^B)]^{S_2} \E[\exp (\frac{1}{S_3}H_n^C)]^{S_3}\\&\times\E[\exp (\frac{1}{S_4}H_n^D)]^{S_4}\E[\exp (\frac{1}{S_5}H_n^E)]^{S_5}
    \end{align*}
    where \[
        S=\sum_{i=0}^{4}n^{i\epsilon_0} \text{ and }S_i= n^{(4-i)\epsilon_0}/S.
    \]
    Proposition \ref{upper_B},\ref{upper_C},\ref{upper_D},\ref{upper_E} imply that if we choose $\epsilon_0$ small enough, we have \[
    \E[\exp (\frac{1}{S_2}H_n^B)]^{S_2} \E[\exp (\frac{1}{S_3}H_n^C)]^{S_3}\E[\exp (\frac{1}{S_4}H_n^D)]^{S_4}\E[\exp (\frac{1}{S_5}H_n^E)]^{S_5}=1+O(n^{-\epsilon'})
    \]
    for some $\epsilon'>0$.
    This implies
    \[
    \E[\exp(H_n(\sigma))]\leq\E[\exp (\frac{1}{S_1}H_n^A)]^{S_1} (1+O(n^{-\epsilon'})).
    \]
    Furthermore, we apply H\"older with different direction and we have 
 \begin{align*}
        \E[\exp ({S_1}H_n^A)]&\leq \E[\exp(H_n(\sigma))]^{S_1}\E[\exp (\frac{-S_1}{S_2}H_n^B)]^{S_2} \E[\exp (\frac{-S_1}{S_3}H_n^C)]^{S_3}\\&\times\E[\exp (\frac{-S_1}{S_4}H_n^D)]^{S_4}\E[\exp (\frac{-S_1}{S_5}H_n^E)]^{S_5}
        \\&= \E[\exp(H_n(\sigma))]^{S_1}(1+O(n^{-\epsilon'}))
    \end{align*}
    These two inequalities imply that 
    \[
        \E[\exp(H_n(\sigma))]= \E[\exp(H_n^A)](1+O(n^{-\epsilon'})).
    \]
     Moreover, since $1/S_1=1+O(n^{-\epsilon_0})$ we can ensure that $H_n^A/S_1$ has same under the threshold of above threshold condition with $H_n^A$. We apply proposition \ref{under_free},\ref{NIM_free} and this proves the theorem except last part about $\alpha<1$. See Appendix \ref{app:prop} for the proof of the refined fluctuation for $\alpha<1$ case.
\end{proof}
Now we can extend the Gibbs measure structure and limiting spin law for NIM model to the heavy-tail model.
\begin{proof}[Proof of Theorem \ref{Gibbs measure}]
    For bounded function $0\leq f\leq 1$, we can apply H\"older inequality again:
        \begin{align*}
        \E[f\exp(H_n(\sigma))]&\leq \E[f\exp (\frac{1}{S_1}H_n^A)]^{S_1}\E[f\exp (\frac{1}{S_2}H_n^B)]^{S_2} \E[f\exp (\frac{1}{S_3}H_n^C)]^{S_3}\\&\times\E[f\exp (\frac{1}{S_4}H_n^D)]^{S_4}\E[f\exp (\frac{1}{S_5}H_n^E)]^{S_5}
        \\& \leq \E[f\exp (\frac{1}{S_1}H_n^A)]^{S_1}\E[\exp (\frac{1}{S_2}H_n^B)]^{S_2} \E[\exp (\frac{1}{S_3}H_n^C)]^{S_3}\\&\times\E[\exp (\frac{1}{S_4}H_n^D)]^{S_4}\E[\exp (\frac{1}{S_5}H_n^E)]^{S_5}
        \\&=\E[f\exp (\frac{1}{S_1}H_n^A)]^{S_1}(1+O(n^{-\epsilon'})).
    \end{align*}
    Similarly, we have the opposite 
     \begin{align*}
        \E[f\exp ({S_1}H_n^A)]&\leq \E[f\exp(H_n(\sigma))]^{S_1}\E[f\exp (\frac{-S_1}{S_2}H_n^B)]^{S_2} \E[f\exp (\frac{-S_1}{S_3}H_n^C)]^{S_3}\\&\times\E[f\exp (\frac{-S_1}{S_4}H_n^D)]^{S_4}\E[f\exp (\frac{-S_1}{S_5}H_n^E)]^{S_5}
        \\&= \E[f\exp(H_n(\sigma))]^{S_1}(1+O(n^{-\epsilon'})).
    \end{align*}
    This shows
    \[
    \E[f\exp(H_n(\sigma))]=\E[f\exp(H_n^A)](1+O(n^{-\epsilon'})).
    \]
    The overlap function, the restricted overlap, the limiting spin value all satisfy this condition. 
    Therefore, this equality, Proposition \ref{Gibbs_upper} and Proposition \ref{under_overlap} prove the theorem.
\end{proof}
This implies the ultrametricity and ultrametricity breaking phase transition.
\begin{proof}[Proof of Theorem \ref{ultrametric}]
    If all interactions are below the threshold, the overlap is zero, which implies ultrametricity.
    For the case where the dominant interaction surpasses the threshold, we express three replicas $\sigma^1,\sigma^2,\sigma^3$ as $\sigma^i=\frac{1}{\sqrt{n}}(\sigma^i_1,...,\sigma^i_p,\sigma^i|)$ where $H\sigma_1...\sigma_p\alpha(p)n^{-(p-2)/2}$ is the dominant interaction in the Hamiltonian. Due to Theorem \ref{Gibbs measure}, we know its geometry and $t$ is the value in the theorem.
    \begin{itemize}
        \item $p=2$: The only possible values for the overlap are just $2t$ and $-2t$. Hence, it satisfies ultrametricity.
        \item $p=3$: The only possible values for the overlap are just $3t$ and $-t$. Hence, it satisfies ultrametricity.
        \item $p\geq 4$: First, we choose just one replica $\sigma^1$.
        Due to Theorem \ref{Gibbs measure}, \[
        \P(|\frac{1}{\sqrt{n}}(-\sigma^1_1,-\sigma^1_2,\sigma^1_3,...,\sigma^1_p)-\frac{1}{\sqrt{n}}(\sigma^2_1,...,\sigma^2_p)|\leq n^{-\epsilon} )> \frac{1}{2^{p-1}}-\epsilon.
        \]
        and \[
        \P(|\frac{1}{\sqrt{n}}(-\sigma^1_1,-\sigma^1_2,-\sigma^1_3,-\sigma^1_4,\sigma_5^1,...,\sigma^1_p)-\frac{1}{\sqrt{n}}(\sigma^3_1,...,\sigma^3_p)|\leq n^{-\epsilon} )> \frac{1}{2^{p-1}}-\epsilon.
        \]
        The probability of choosing three replicas as above is $>(\frac{1}{2^{p-1}}-\epsilon)^2>0$ and for this case the overlaps are
        \[
        R_{1,2 }\text{ and } R_{2,3}=t(p-4)+O(n^{-\epsilon}),\;  R_{1,3}= t(p-8)+O(n^{-\epsilon}).
        \]
        For this case, $R_{1,3}<\min(R_{1,2},R_{2,3})$ holds with probability $>(\frac{1}{2^{p-1}}-\epsilon)^2>0$. This shows the breakdown of ultrametricity.
    \end{itemize}
\end{proof}
\begin{proof}[Proof of Theorem \ref{prob_land}]
We already know that $|H_{1,p}|$ for each $p$ converges to a Fr\'echet distribution $X_p$. Hence, finding the dominant interaction is to find $p$ to maximize $f_p(\beta\alpha(p) X_p)$. The limiting Gibbs measure and the free energy are just functions of i.i.d Fr\'echet random variables $X_p$. Therefore, Theorem \ref{Free} and Theorem \ref{Gibbs measure} easily prove this theorem. Furthermore, the probability of the event $\mathcal{F}_1$ converges to
\[
    \P(\mathcal{F}_1)= \prod_{p\geq 2} \P(|\beta \alpha(p)X_p|<H_p^*)
\]
and the free energy converges in distribution to 
\[
F_n\to \max_p f_p(\beta \alpha(p)X_p).
\]

\end{proof}

\appendix
\section{Proofs of Technical Lemmas}\label{app:proof}
\begin{proof}[Proof of Lemma \ref{integral}]
Define
\[
J_{i_1,\dots,i_k}(r)=\int_{\{\sigma\in\mathbb{R}^n: \sum_{i=1}^n \sigma_i^2=r^2\}}\sigma_1^{\,i_1}\cdots \sigma_k^{\,i_k}\,\frac{1}{(2\pi)^{n/2}} \exp\Bigl(-\frac{1}{2}\sum_{i=1}^n\sigma_i^2\Bigr)\,d\sigma.
\]
A change of variable shows that
\[
J_{i_1,\dots,i_k}(r)= r^{\,n-1+i_1+\cdots+i_k}\exp\Bigl(-\frac{r^2-1}{2}\Bigr)J_{i_1,\dots,i_k}(1).
\]
Thus, integrating with respect to $r$, we have
\[
J_{i_1,\dots,i_k}(1)\,\frac{e^{1/2}}{2}\,\mathbb{E}\Bigl[|X|^{\,n-1+i_1+\cdots+i_k}\Bigr]
=\int_0^\infty r^{\,n-1+i_1+\cdots+i_k}\exp\Bigl(-\frac{r^2-1}{2}\Bigr)J_{i_1,\dots,i_k}(1)\,dr
=\int_0^\infty J_{i_1,\dots,i_k}(r)\,dr.
\]
On the other hand, by a standard computation this integral equals
\[
\prod_{t=1}^k \mathbb{E}\Bigl[X^{\,i_t}\Bigr].
\]
Finally, observing that the expectation of $\sigma_1^{i_1}\cdots \sigma_k^{i_k}$ under the uniform measure on the sphere is given by
\[
\mathbb{E}\bigl[\sigma_1^{i_1}\cdots \sigma_k^{i_k}\bigr]=\frac{J_{i_1,\dots,i_k}(\sqrt{n})}{J_{0}(\sqrt{n})},
\]
we deduce that
\[
\mathbb{E}\bigl[\sigma_1^{i_1}\cdots \sigma_k^{i_k}\bigr]
=\frac{n^{\frac{1}{2}\sum_{t=1}^k i_t}\,\mathbb{E}[|X|^{n-1}]\,\mathbb{E}[X^{i_1}]\cdots\mathbb{E}[X^{i_k}]}
{\mathbb{E}\bigl[|X|^{\,i_1+\cdots+i_k+n-1}\bigr]}.
\]
Similarly, we have 
\[
\mathbb{E}\bigl[|\sigma_1|^{i_1}\cdots |\sigma_k|^{i_k}\bigr]
=\frac{n^{\frac{1}{2}\sum_{t=1}^k i_t}\,\mathbb{E}[|X|^{n-1}]\,\mathbb{E}[|X|^{i_1}]\cdots\mathbb{E}[|X|^{i_k}]}
{\mathbb{E}\bigl[|X|^{\,i_1+\cdots+i_k+n-1}\bigr]}.
\]
\end{proof}
\begin{proof}[Proof of Lemma \ref{exp_basic}, \ref{phase3}, \ref{phase2}]
Its Taylor expansion implies
\[
    \E(\exp( H \sigma_1...\sigma_p n^{-(p-2)/2}))=\sum_{l=0}^{\infty} \frac{1}{(2\ell)!} \frac{\E\{X^{2\ell}\}^p\E\{|X|^{n-1}\}}{\E\{|X|^{2\ell p+n-1}\}}(Hn)^{2\ell}.
\]
We define $\widehat{H}= H \sigma_1...\sigma_p n^{-(p-2)/2} $.
\\Since the moment of the standard normal variable is
\[
\E[|X|^k] = \frac{2^{k/2}}{\sqrt{\pi}} \Gamma\left(\frac{k+1}{2}\right),
\]
and 
\[
    \ell!= \Gamma(\ell+1),
\]
if we define
\[
    f(\ell)=\log (\frac{( |H n|)^{2\ell}}{\Gamma(2\ell+1)}\frac{\Gamma(\ell+\frac{1}{2})^{p}\Gamma(\frac{n}{2})}{\pi^{p/2}\Gamma(\ell p+\frac{n}{2})})
\]
for non-negative real number $l$,then it is well defined and also it satisfies 
\[
    f(\ell)=\E[\frac{|\widehat{H}|^{2\ell}}{(2\ell)!}]
\]
$l \in \mathbb{N}$ or $\mathbb{N}+\frac{1}{2}$.
\\Furthermore, Stirling's approximation implies
\[
    \log \Gamma(z)= z\log z -z+ \frac{1}{2}\log \frac{2\pi}{z}+O(\frac{1}{z})
\]
and 
\[
    \log \Gamma(z+1)=z\log z-z+\frac{1}{2}\log 2\pi z +O(z^{-1}).
\]
This also gives its \( k \)-th moment approximation
\[
    \log\E[|X|^k]= \frac{k}{2}\log(k+1)+\frac{1}{2}\log 2-\frac{k+1}{2}+O(\frac{1}{k}).
\]
 Then this gives asymptotics for each term of the expansion as \begin{align*}
    f(\ell)&=\log (\frac{( H n)^{2\ell}}{(2\ell)!}\frac{\E\{X^{2\ell}\}^p\E\{|X|^{n-1}\}}{\E\{|X|^{2\ell p+n-1}\}})\\&= 2\ell \log(H n )-(2\ell \log(2l)-2\ell+\frac{1}{2}\log(4\pi \ell))\\&+p\{\ell \log(2\ell+1)+\frac{1}{2}\log2 -\frac{2\ell+1}{2}+O(\ell^{-1})\}+\frac{n-1}{2}\log n +\frac{1}{2}-\frac{n}{2}+O(n^{-1})
    \\&-(\frac{2\ell p+n-1}{2}\log(2\ell p+n)+\frac{1}{2}\log2-\frac{2\ell p+n}{2}),
\end{align*}
and 
\[
    f(\ell+1)-f(\ell)=2\log(Hn)- \log(2\ell+1)(2\ell+2)+p\log(2\ell+1)-\log(2\ell p+n+2\cdot 1)\cdots(2\ell p+n+2(p-1)).
\]
If $\ell\ll n$, this function is always negative, which means $f$ is decreasing in that region.
If we put $\ell$ to $cn$, we have
\begin{align*}
        &f(cn)= 2cn \log( H n )-(2cn \log(2cn)-2cn+\frac{1}{2}\log(4\pi cn)) \\&+p\{cn \log(2cn+1)+\frac{1}{2}\log2 -\frac{2cn+1}{2}\}+\frac{n-1}{2}\log n +\frac{1}{2}-\frac{n}{2}
    \\&-(\frac{2pcn+n-1}{2}\log(2pcn+n)+\frac{1}{2}\log2-\frac{2pcn+n}{2})+O(n^{-1}).
\end{align*}
We have asymptotic
\[ 
    cn \log(2cn+1)= cn\log(2cn)+ \frac{1}{2}+O(n^{-1}),
\]
and putting it above and simplifying, we have
\begin{align*}
       f(cn)&=n\{2c\log(H) -2c\log 2c +2c +pc \log 2c -\frac{2pc+1}{2}\log(2pc+1)\}-\frac{1}{2}\log n 
       \\&+\{\frac{p-1}{2}\log 2-\frac{1}{2}\log(\pi c)+\frac{1}{2}+\frac{1}{2}\log(2pc+1)\}+O(n^{-1}),
\end{align*}
and let $g(c)$ be the coefficient of the \( n \)-th term 
\[
    g(c)=2c\log( H)-2c\log(2c)+2c+pc\log 2c -\frac{2pc+1}{2}\log(2pc+1).
\]
We express other remaining terms except $g(c)n$ in $f(cn)$ as $\square$. Then, we have 
\[
    f(cn)=g(c)n+ \square.
\]
For large \( n \), the most influential term is $g(c)$. 
Consider its derivative
\[
    g'(c)=2\log( H) +(p-2)\log (2c)-p\log(2pc+1)    
\]
and second derivative 
\[
    g''(c)= \frac{p-2-4pc}{c(2pc+1)}.
\]
For $p\geq 3$, $g'$ has its maximum at $c=\frac{p-2}{4p}$ and 
\[
    g'(\frac{p-2}{4p})=2\log( H)+(p-2)\log (\frac{p-2}{2p})-p\log(\frac{p}{2})=\log((2 H)^2\frac{(p-2)^{p-2}}{p^{2p-2}}).
\]
Here we have two cases:
\begin{itemize}
    \item If $g'(\frac{p-2}{4p})<0$, then \( g \) is a decreasing function and $\ell=0,1$ terms become the dominant terms. 
    \item If $g'(\frac{p-2}{4p})>0$, then $g'(h)=0$ has two solutions $h_{-}< \frac{p-2}{4p}< h_{+}$ and \( g \) has its maximum at $h_+$.
\end{itemize}
For the second case, we need to distinguish further
\begin{itemize}
    \item  If $g(h_+)<0$, then $\ell=0,1$ term is dominant as above.
    \item   If $g(h_{+})>0$, 
\[
e^{ng(h_{+})}<\E(\exp( H \sigma_1...\sigma_p n^{-(p-2)/2}))<n((h_{+}+1))e^{ng(h_{+})}+O(1)e^{ng(h_{+})}
\]
and we have
\[
    \frac{1}{n}\log \E(\exp( H \sigma_1...\sigma_p n^{-(p-2)/2}))= g(h_{+})+O(\frac{\log n }{n}).
\]
\end{itemize}
This shows phase transition between as the value of $g(h_+)$. Here we can define 
$\lambda_p(H)$ to be $h_+$ and the solution of $g(h_+)=0$ to be $H_p^*$ and $f_p$ to be $g(h_+)$.
Now we prove remaining properties in the regime $g(h_+)>0$.
We put $c=h_+$ and recall the expression
\(f(cn)=ng(c)+\square\).
We utilize Taylor's approximation to here and we have
\begin{align*}
    f(cn+\triangle)&= ng(c+\frac{\triangle}{n})+\square
    \\&=n(g(c)+\frac{\triangle}{n}g'(c)+(\frac{\triangle}{n})^2g''(c)+O(g'''(c)(\frac{\triangle}{n})^3))+\square
    \\&= f(cn)+\triangle g'(c)+\frac{\triangle^2}{n}g''(c)+O(\frac{\triangle^3}{n^2}).
\end{align*}
Since we have $g'(c)= O(1/n)$ and $g''(c)=O(1)$,
we have 
\[
    f(cn+\triangle)= f(cn)+ O(\frac{\triangle}{n})+g''(c)\frac{\triangle^2}{n}
\]
for $\triangle\ll n$.
This gives the asymptotics 
\[
    \sum_{\triangle=-m}^m \exp(f(cn+\triangle))=\exp(f(cn))\sum_{\triangle=-m}^m \exp(g''(c)\frac{\triangle^2}{n}+O(\frac{\triangle}{n})).
\]
For $\sqrt{n}\ll m\ll n $,
this approximation is \[\sqrt{\frac{2\pi n}{g''(c)}}\exp(f(cn)).\]
Other remaining part is small enough and negligible. 
Moreover, we have 
\begin{align*}
    &\sum_{\triangle=-m}^{m}[\exp (f(cn+\triangle+1/2))-\exp(f(cn+\triangle))]
    \\&\;\; =\sum_{\triangle=-m}^{m}\exp (f(cn+\triangle))(\exp f'(cn+t_{\triangle} )-1)
    \\& \;\;= O(\frac{m}{n})\sum_{\triangle=-m}^{m}\exp (f(cn+\triangle)).
\end{align*}
For $p=2$, the phase transition location is quite different. We consider \[
    g(c)=2c\log( H)+2c -\frac{4c+1}{2}\log(4c+1)
\] and 
\[
    g'(c)= 2\log(H)-2\log (4c+1) 
\]
We have two cases:
\begin{itemize}
    \item For $H<1$,The function $g$ is decreasing and
    \begin{align*}
    \E(\exp( H \sigma_1\sigma_2 ))&=\sum_{\ell=0}^{\infty} \frac{1}{(2\ell)!} \frac{\E\{X^{2\ell}\}^2\E\{|X|^{n-1}\}}{\E\{|X|^{4\ell+n-1}\}}(Hn)^{2\ell}\\&= \sum_{\ell=0}^{\infty}4^{-\ell}\binom{2\ell}{\ell} H^{2\ell}(1+O(n^{-\epsilon}))=(1-H^2)^{-1/2}(1+O(n^{-\epsilon})).
\end{align*}
    \item For $H>1$, $g$ has its maximum at $c= \frac{H-1}{4}$.
\end{itemize}
The maximum value is 
\[
g(c)= 2c-\frac{1}{2}\log(1+2pc)=\frac{ H-1}{2}-\frac{1}{2}\log(H).
\]
We have similar approximation using Taylor expansion:
\begin{align*}
    &f(cn+\triangle)=f(cn)+O(\frac{\triangle}{n})+g''(c)\frac{\triangle^2}{n}.
\end{align*}
For $m\gg \sqrt{n}$, we have
\[
    \sum_{\triangle=-m}^m \exp (f(cn+\triangle))\approx \sqrt{\frac{2\pi n }{g''(c)}}\exp(f(cn))
\]
and other reamining part is negligible.
Moreover, we have
\[
    \sum_{\triangle=-m}^m [\exp((f(cn+\triangle+\frac{1}{2}))-\exp((f(cn+\triangle))]= O(\frac{m}{n})\sum_{\triangle=-m}^m \exp(f(cn+\triangle)).
\]
Furthermore, in any case of $p$, for the maximum point $cn$, we have
\[
    g(c)= 2c- \frac{1}{2}\log(1+2pc).
\]
\end{proof}
\begin{proof}[Proof of Lemma \ref{convexity1}]
Define
\[
f(z) \;=\; p\log \Gamma\!\Bigl(\tfrac{z+1}{2}\Bigr)\;-\;\log \Gamma(z+1).
\]
Recall that
\[
\frac{d}{dx}\,\log \Gamma(x) \;=\; \psi(x)\quad(\text{the digamma function}).
\]
Then
\[
f'(z) \;=\; \frac{p}{2}\,\psi\!\Bigl(\tfrac{z+1}{2}\Bigr)\;-\;\psi(z+1),
\]
and
\[
f''(z) \;=\; \frac{p}{4}\,\psi'\!\Bigl(\tfrac{z+1}{2}\Bigr)\;-\;\psi'(z+1).
\]
We wish to show that 
\[
f''(z) \;>\; 0.
\]

Using the duplication formula
\[
\Gamma(2z) \;=\; 2^{2z-1}\,\sqrt{\pi}\,\Gamma(z)\,\Gamma\!\Bigl(z+\tfrac{1}{2}\Bigr),
\]
one obtains
\[
4\,\psi(2z) \;=\; \psi(z)\;+\;\psi\!\Bigl(z+\tfrac{1}{2}\Bigr).
\]
Taking derivatives yields
\[
\psi'\!\Bigl(\tfrac{z+1}{2}\Bigr)\;+\;\psi'\!\Bigl(\tfrac{z}{2}+1\Bigr)
\;=\;4\,\psi'(z+1).
\]
Hence
\[
\frac{1}{2}\,\psi'\!\Bigl(\tfrac{z+1}{2}\Bigr)\;-\;\psi'(z+1)
\;=\;\frac{1}{4}\,\Bigl(\,\psi'\!\Bigl(\tfrac{z+1}{2}\Bigr)\;-\;\psi'\!\Bigl(\tfrac{z}{2}+1\Bigr)\Bigr)
\;>\;0,
\]
since $\psi'$ is strictly decreasing.  Therefore,
\[
g''(z)=f''(z) \;>\; 0.
\]
\end{proof}
\begin{proof}[Proof of Lemma \ref{decom_1}]
For standard random variable X, the left side is
    \[
        \E[\sigma_{I_1}^{a_1}...\sigma_{I_t}^{a_t}]=\frac{\E[X^{a_1}]^{|I_1|}\cdots \E[X_t^{a_t}]^{|I_t|}\E[|X|^{n-1}]}{\E[|X|^{n-1+\sum a_i|I_i|}]},
    \]
    and the right side is
    \[
        \prod_{i=1}^{t}\E[\sigma_{I_i}^{a_i}]= \prod_{i=1}^{t}\frac{\E[X^{a_i}]^{|I_i|}|X|^{n-1}}{\E[|X|^{n-1+a_i|I_i|}]}.
    \]
Since
\[
    \frac{\E[|X|^{n-1}]}{\E[|X|^{n-1+\sum a_i|I_i|}]}\leq \prod_{i=1}^t \frac{\E[|X|^{n-1}]}{\E[|X|^{n-1+a_i|I_i|}]},
\]
the right side is greater than equal to the left side. 
\end{proof}
\begin{proof}[Proof of Lemma \ref{decompose}]
Applying above lemma and Taylor expansion, we obtain
    \begin{align*}
        \E[\exp(\sum H_I\sigma_I)]&=\E[\prod\exp(H_I\sigma_I))]=\sum_{k_I\in Z_{\geq 0}}\E[\prod \frac{(H_I\sigma_I)^{k_I}}{k_I!}]\\&\leq\sum_{k_I\in Z_{\geq 0}}\prod \frac{\E[(H_I\sigma_I)^{k_I}]}{k_I!}=\prod \E(\exp(H_I\sigma_I)).
    \end{align*}
\end{proof}
\section{Refined fluctuation in the strongly heavy-tailed regime \texorpdfstring{$\alpha<1$}{alpha<1}}\label{app:prop}
We prove the last special case of Theorem \ref{Free} for $\alpha<1$.
Here we recall the definition $\beta(p)$ which is the smallest $p$ such that $\alpha(p)\neq 0$. 
\begin{prop}\label{prop:under_free}
    For under the threshold regime, if $\alpha<1$ and $\beta(p)\geq 3$ holds, the free energy shows  \[
            n^{p-2}\log Z_n|\mathcal{F}_1= \frac{1}{2}\beta^2\alpha(\beta(p))^2\sum_{|I|=\beta(p)} H_{I,p}^2+O(n^{-\epsilon}).
        \]
\end{prop}
\begin{proof}
We define $H_n^{BCD}=H_n^B+H_n^C+H_n^D$ and $p$-spin part to be $H_{n,p}^{BCD}$ 
    \begin{align*}
        \E[\exp(sH_n^{BCD})]&\leq \prod_{p=2}^{M\log n}\E[\exp(sM\log n H_{n,p}^{BCD})]^{1/M\log n}\\ & \leq \max_p\E[\exp(sM\log n H_{n,p}^{BCD})].
     \end{align*}
    Furthremore, 
    \begin{align*}
        \E[\exp(H+H_I\sigma_I+H_J\sigma_J)]&\leq \E[\exp(H+(H_I+H_J)\sigma_I)]^{H_I/(H_I+H_J)}\\ & \times \E[\exp(H+(H_I+H_J)\sigma_J)]^{H_J/(H_I+H_J)}
        \\ &\leq \max_{K=I\text{ or }J} \E[\exp(H+(H_I+H_J)\sigma_K)].
    \end{align*}
    Applying this inequality and we have 
    \[
    \E[\exp(sH_{n,p}^{BCD})]\leq \E[\exp((s\sum_{i>n^{\epsilon} }H_{i,p})\sigma_1...\sigma_p)]=1+n^{2-p}s^2(\sum_{i>n^{\epsilon}} H_{i,p} )^2.
    \]
    Lemma \ref{heavy_upper} implies that we have 
    $\sum_{n^{\epsilon}\leq i} H_{i,p}\leq n^{\epsilon_0}i^{-1/\alpha}\leq n^{\epsilon(1-1/\alpha)+\epsilon_0}$.
    We can choose $\epsilon_0<\frac{1}{2}(1/\alpha-1)\epsilon$ and $\epsilon_1$ small enough that $s\sum_{i>n^{\epsilon}} H_{i,p}<n^{-\epsilon_2}$ holds for all $s<n^{\epsilon_1}$ and for some $\epsilon_2>0$.
    This proves
     \[
    \E[\exp(sH_{n}^{BCD})]<1+n^{2-p-\epsilon_2}
    \]
    for $s<n^{\epsilon_1}$.
    We apply H\"older inequality to our original Hamiltonian as 
    \[
    \E[\exp (H_n(\sigma))]\leq \E[\exp(T_1H_n^A)]^{1/T_1}\E[\exp(T_2 H_n^{BCD})]^{1/T_2}\E[\exp(T_3 H_n^E)]^{1/T_3}
    \]
    and
    \[
     \E[\exp(\frac{1}{T_1}H_n^A)]\leq\E[\exp (H_n(\sigma))]^{1/T_1}\E[\exp(-\frac{T_2}{T_1} H_n^{BCD})]^{1/T_2}\E[\exp(-\frac{T_3}{T_1} H_n^E)]^{1/T_3}
    \]
    where $T_1= 1/T$, $T_2=n^{\epsilon_3}/T$, $T_3= n^{2\epsilon_3}
    /T$ and $T=1+n^{\epsilon_3}+n^{2\epsilon_3}$.
    If we choose $\epsilon_3<\epsilon_2/3$, this and Proposition \ref{upper_E} imply 
    \[
    \E[\exp(H_n(\sigma))]= \E[\exp H_n^A](1+O(n^{2-p-\epsilon_2})).
    \]
     This inequality and Proposition \ref{NIM_free} prove the refined fluctuation for $\alpha<1$ case. 
\end{proof}
\subsection*{Acknowledgement}
 The author is deeply grateful to Ji Oon Lee for invaluable guidance, unfailing encouragement, and many stimulating discussions. Special thanks go to Gérard Ben Arous, whose patient instruction in spin glass theory and pivotal suggestion to extend the study from the pure p-spin to the mixed p-spin model were essential to this work. The author also thanks Michel Talagrand for insightful comments that greatly improved the exposition. This work was partially supported by National Research Foundation of Korea under grant number NRF-2019R1A5A1028324 and NRF-2023R1A2C1005843.
\providecommand{\bysame}{\leavevmode\hbox to3em{\hrulefill}\thinspace}
\providecommand{\noopsort}[1]{}
\providecommand{\mr}[1]{\href{http://www.ams.org/mathscinet-getitem?mr=#1}{MR~#1}}
\providecommand{\zbl}[1]{\href{http://www.zentralblatt-math.org/zmath/en/search/?q=an:#1}{Zbl~#1}}
\providecommand{\jfm}[1]{\href{http://www.emis.de/cgi-bin/JFM-item?#1}{JFM~#1}}
\providecommand{\arxiv}[1]{\href{http://www.arxiv.org/abs/#1}{arXiv~#1}}
\providecommand{\doi}[1]{\url{https://doi.org/#1}}
\providecommand{\MR}{\relax\ifhmode\unskip\space\fi MR }
\providecommand{\MRhref}[2]{%
  \href{http://www.ams.org/mathscinet-getitem?mr=#1}{#2}
}
\providecommand{\href}[2]{#2}

\end{document}